\def\isarxiv{1}

\documentclass[11pt]{article}
\usepackage[margin=1in]{geometry}
\usepackage{latexsym, amscd, amsfonts, mathrsfs, amsmath, amssymb, amsthm}
\usepackage{bm}
\usepackage{hyperref,color}
\usepackage[noend]{algpseudocode}
\usepackage[capitalize,nameinlink]{cleveref}
\usepackage{caption}
\usepackage{subcaption}

\makeatletter
\newcounter{algorithmicH}
\let\oldalgorithmic\algorithmic
\renewcommand{\algorithmic}{%
  \stepcounter{algorithmicH}
  \oldalgorithmic}
\renewcommand{\theHALG@line}{ALG@line.\thealgorithmicH.\arabic{ALG@line}}
\makeatother

\usepackage[dvipsnames]{xcolor}
\hypersetup{
	colorlinks=true,
	pdfpagemode=UseNone,
    citecolor=OliveGreen,
    linkcolor=NavyBlue,
    urlcolor=Magenta,
	pdfstartview=FitW
}
\usepackage{stmaryrd, tikz-cd, enumitem, mathrsfs, bbm, algorithm, url, esint, todonotes, listings, moreverb, makecell}

\def\bE{\mathbb{E}}

\def\bP{\mathbb{P}}
\def\bR{\mathbb{R}}

\def\bZ{\mathbb{Z}}

\def\bfb{\mathbf{b}}

\def\cA{\mathcal{A}}
\def\cB{\mathcal{B}}

\def\cI{\mathcal{I}}

\def\cW{\mathcal{W}}
\def\cX{\mathcal{X}}
\def\cY{\mathcal{Y}}
\def\cZ{\mathcal{Z}}

\def\scrB{\mathscr{B}}
\def\scrP{\mathscr{P}}
\def\scrT{\mathscr{T}}

\def\wt{\widetilde}
\def\wh{\widehat}
\DeclareMathOperator\br{\mathrm{br}}

\DeclareMathOperator\Ber{\mathrm{Ber}}
\DeclareMathOperator\BEC{\mathrm{BEC}}
\DeclareMathOperator\BSC{\mathrm{BSC}}
\DeclareMathOperator\BOHT{\mathrm{BOHT}}
\DeclareMathOperator\BP{\mathrm{BP}}
\DeclareMathOperator\Id{\mathrm{Id}}
\DeclareMathOperator\sgn{\mathrm{sgn}}

\DeclareMathOperator\HSBM{\mathrm{HSBM}}
\DeclareMathOperator\Unif{\mathrm{Unif}}
\DeclareMathOperator\Pois{\mathrm{Pois}}
\DeclareMathOperator\poly{\mathrm{poly}}
\DeclareMathOperator\supp{\mathrm{supp}}

\DeclareMathOperator\TV{\mathrm{TV}}

\def\cond{\mathrm{cond}}
\def\KS{\mathrm{KS}}

\def\iidsim{\stackrel{\text{i.i.d.}}{\sim}}

\newcommand\setItemnumber[1]{\setcounter{enumi}{\numexpr#1-1\relax}}

\newcommand{\rom}[1]{\textup{\uppercase\expandafter{\romannumeral#1}}}

\newcommand{\txtsf}[1]{{\normalfont\textsf{#1}}}
\newcommand{\txtsc}[1]{{\normalfont\textsc{#1}}}

\newtheorem{theorem}{Theorem}
\newtheorem{lemma}[theorem]{Lemma}
\newtheorem{proposition}[theorem]{Proposition}
\newtheorem{corollary}[theorem]{Corollary}

\theoremstyle{definition}
\newtheorem{definition}[theorem]{Definition}

\begin{document}
\title{Exact reconstruction thresholds on hypertrees over a symmetric binary alphabet}
\ifdefined\isarxiv
\author{Yuzhou Gu\thanks{\texttt{sevenkplus.g@gmail.com}}}
\else
\author{Anonymous Authors}
\fi
\date{}

\maketitle

\begin{abstract}
  We establish the exact reconstruction thresholds for a class of broadcasting models on hypertrees over a symmetric binary alphabet. As a consequence, we show that the condensation threshold coincides with the Kesten-Stigum threshold for random NAE-SAT and random hypergraph bicoloring with arity at most four at any temperature, confirming a prediction of Ricci-Tersenghi et al.~'19. We also determine the exact weak recovery threshold for the two-community hypergraph stochastic block model in certain parameter regimes.

  The proof combines an improved robust non-reconstruction analysis with a computer-assisted rigorous population dynamics algorithm. It relies heavily on information-theoretic tools, in particular the theory of binary memoryless symmetric channels and channel comparison methods.
\end{abstract}


\section{Introduction} \label{sec:intro}

\subsection{Broadcasting on hypertrees} \label{sec:intro:boht}
Consider the following broadcasting on hypertrees (BOHT) model.
This model has three parameters $d\in \bR_{\ge 0}$, $r\in \bZ_{\ge 2}$, and $\lambda\in \left[ -\frac 1{2^{r-1}-1}, 1\right]$, and is denoted $\BOHT(\Pois(d),r,\lambda)$.
We first generate an $r$-uniform Galton-Watson tree $T$ in which every vertex has $\Pois(d)$ downward hyperedges.
Then a label $\sigma_u\in \{\pm\}$ is generated for every vertex $u\in T$ as follows. The root $\rho$ has label $\sigma_\rho \sim \Unif(\{\pm\})$.
For every vertex $u$, conditioned on $\sigma_u$, for every downward hyperedge $(u,v_1,\ldots,v_{r-1})$, the children's labels are generated according to $(\sigma_{v_1},\ldots,\sigma_{v_{r-1}}) \sim B(\cdot | \sigma_u)$, where $B=B_{r,\lambda}: \{\pm\} \to \{\pm\}^{r-1}$ is the channel
\begin{align} \label{eqn:b-r-lambda}
  B_{r,\lambda}(y_1,\ldots,y_{r-1} | x) = \left\{
    \begin{array}{ll}
      \lambda + \frac{1-\lambda}{2^{r-1}} & \text{if}~y_1=\cdots=y_{r-1}=x,\\
      \frac{1-\lambda}{2^{r-1}} & \text{o.w.}
    \end{array}
  \right.
\end{align}
This completes the labeling of $T$.
For $r=2$, this is the Ising model on a tree.

The reconstruction problem is a fundamental problem for the BOHT model.
Let $L_k$ be the set of vertices at distance $k$ from $\rho$, and let $T_k$ be the structure of the tree induced by the vertices at distance at most $k$ from $\rho$.
The reconstruction problem asks whether the boundary labels carry non-vanishing information about the root:
\begin{align}
  \lim_{k\to \infty} I(\sigma_\rho ; T_k,\sigma_{L_k}) > 0.
\end{align}
When the limit is positive we say reconstruction is possible; otherwise reconstruction is impossible.

The reconstruction problem for the BOHT model can be studied through a belief propagation (BP) operator on the space of information channels.
For $k\ge 0$, let $M_k$ denote the channel $\sigma_\rho \mapsto (T_k, \sigma_{L_k})$.
Reconstruction is possible if and only if
\begin{align}
  \lim_{k\to \infty} I(\pi, M_k) > 0,
\end{align}
where $\pi = \Unif(\{\pm\})$.
These channels evolve by the BP operator: for a binary-input channel $P$,
\begin{align} \label{eqn:intro:bp-operator}
  \BP(P) = \bE_{b\sim \Pois(d)} \left( P^{\times (r-1)} \circ B \right)^{\star b},
\end{align}
where $\bE$ denotes channel mixture and $(\cdot)^{\star b}$ denotes $\star$-power (repeated observation).
(See \cref{sec:prelim:bms} for relevant information theory background.)
The trivial channel $P = 0$ is always a fixed point, and the reconstruction problem is equivalent to asking whether the BP operator admits a non-trivial fixed point. For the special BOHT model, and more generally for binary symmetric BOHT models, it suffices to work with binary memoryless symmetric (BMS) channels, which have a simpler representation than general binary-input channels.

The Kesten-Stigum (KS) threshold for BOHT is $(r-1) d \lambda^2 = 1$.
A folklore result (generalizing \cite{kesten1966additional}) states that reconstruction is possible above the KS threshold.
For the Ising model ($r=2$), \cite{bleher1995purity,evans2000broadcasting} proved that the reconstruction threshold matches the KS threshold.
For the special BOHT model, \cite{gu2023weak} proved that KS is tight for $r=3,4$, $\lambda>0$, and any $r$, $\lambda\ge \frac 15$; and that KS is not tight for $r\ge 7$ and $|\lambda|$ small enough. \cite{gu2024community} proved that KS is tight for $r\le 6$ and $|\lambda|$ small enough.
\cite{yuan2022testing,gu2024community} showed that for every $r\ge 5$, there exists $\lambda_r^* \in \left[ -\frac 1{2^{r-1}-1}, 1 \right]$ such that KS is not tight for $\lambda \le \lambda_r^*$. Furthermore, $\lambda_r^* < 0$ for $r=5,6$ and $\lambda_r^*  > 0$ for $r\ge 7$.

Our main result fully establishes tightness of the KS threshold for $r=3,4$, and substantially improves the parameter regimes in which KS is known to be tight for $r=5,6$.
\begin{theorem}[Reconstruction threshold for the special BOHT] \label{thm:boht}
  Consider the special BOHT model $\BOHT(\Pois(d),r,\lambda)$ where $r\in \bZ_{\ge 3}$, $d\in \bR_{\ge 0}$, $\lambda\in \left[-\frac 1{2^{r-1}-1}, 1\right]$.
  If $(r-1) d \lambda^2 \le 1$ and one of the following is true, then reconstruction is impossible.
  \begin{enumerate}[label=(\roman*)]
    \item\label{item:thm:boht:i} $r=3$ or $r=4$;
    \item\label{item:thm:boht:ii} $r=5$ and $\lambda\ge -\frac 1{18}$;
    \item\label{item:thm:boht:iii} $r=6$ and $\lambda\ge \frac 1{46}$.
  \end{enumerate}
\end{theorem}
In comparison, the previous best result for $r=5$ (respectively, $r=6$) \cite{gu2023weak} establishes tightness for $\lambda\ge \frac 17$ (respectively, $\lambda\ge \frac {11}{75}$).
The bounds $-\frac 1{18}$ and $\frac 1{46}$ in \cref{thm:boht} are not tight and can be slightly improved with more computational power. However, we do not determine the smallest $\lambda_5^*$ (respectively, $\lambda_6^*$) for which KS is tight for all $\lambda\ge \lambda_5^*$ (respectively, $\lambda\ge \lambda_6^*$).

\subsection{Applications} \label{sec:intro:app}
\paragraph{Condensation thresholds for random CSPs}
Our first application concerns the condensation threshold for random NAE-SAT and random hypergraph bicoloring.
In random constraint satisfaction problems (CSPs), the condensation threshold is a predicted transition point in the partition function (number of solutions) and in the geometry of the solution space. Determining this threshold is an important problem in the study of random CSPs. We refer the reader to \cite{sly2018phase} for a survey of predicted phase transitions in random CSPs and progress on rigorously establishing them.

In an NAE-$k$-SAT problem, the state space is $\cX = \{\pm\}$, and every clause has the form $e=(\{i_1,\ldots,i_k\}, c_1,\ldots,c_k)$, where $\{i_1,\ldots,i_k\} \in \binom{[n]}{k}$ and $c_1,\ldots,c_k \in \{\pm\}$.
An assignment $x\in \{\pm\}^n$ satisfies a clause $(\{i_1,\ldots,i_k\}, c_1,\ldots,c_k)$ if and only if $\{c_j x_{i_j} : j\in [k]\} = \{\pm\}$.
The random NAE-$k$-SAT model is obtained by choosing $m\sim \Pois(\alpha n)$, and then choosing $m$ clauses independently with $\{i_1,\ldots,i_k\} \sim \Unif\left(\binom{[n]}k\right)$ and $c_1,\ldots,c_k\iidsim \Unif(\{\pm\})$.
The random $k$-uniform hypergraph bicoloring model is obtained by choosing $m$ and $\{i_1,\ldots,i_k\}$ in the same way but taking $c_j=+$ for all $j\in [k]$.

Let $\beta\in \left(0,\infty\right]$ be the inverse temperature parameter.
For an NAE-SAT or hypergraph bicoloring instance, define the partition function
\begin{align}
  Z = \sum_{x\in \{\pm\}^n} \exp\left( -\beta \sum_{e\in E} \mathbbm1\{ c_1(e) x_{i_1(e)} = \cdots = c_k(e) x_{i_k(e)} \}\right),
\end{align}
where $E$ is the set of clauses.
At zero temperature ($\beta=\infty$), $Z$ is the number of satisfying assignments.
The condensation threshold $d_{k,\cond}(\beta)$ is the critical value of $d=k\alpha$ at which the free energy density $\lim_{n\to \infty} \frac 1n \log Z$ ceases to converge to the replica symmetric formula
\begin{align} \label{eqn:nae-sat-rs-pred}
  \log 2 + \frac dk \log\left(1-\frac{1-\exp(-\beta)}{2^{k-1}}\right).
\end{align}

For a large class of random CSPs, including random NAE-SAT and random hypergraph bicoloring, \cite{coja2017information} identified the condensation threshold $d_{k,\cond}(\beta)$ through an infinite-dimensional optimization problem, and showed that it suffices to optimize over fixed points of the belief propagation operator. \Cref{thm:boht} rules out non-trivial fixed points in the regimes below, and therefore implies the following result.
\begin{theorem}[Condensation threshold for random NAE-SAT and random hypergraph bicoloring] \label{thm:nae-sat}
  Consider the random NAE-$k$-SAT model or the random $k$-uniform hypergraph bicoloring model at inverse temperature $\beta \in \left(0, \infty\right]$.
  If one of the following is true, then the condensation threshold is given by
  \begin{align}
    d_{k,\cond}(\beta) = \frac 1{k-1} \left( \frac{2^{k-1}-1+\exp(-\beta)}{1-\exp(-\beta)} \right)^2.
  \end{align}
  \begin{enumerate}[label=(\roman*)]
    \item\label{item:thm:nae-sat:i} $k=3$ or $k=4$;
    \item\label{item:thm:nae-sat:ii} $k=5$ and $\beta \le \log \frac{19}{3}$.
  \end{enumerate}
\end{theorem}
For random hypergraph bicoloring, \cite{ricci2019typology} predicted tightness of the KS threshold for $k=3,4$. \Cref{thm:nae-sat} gives a rigorous confirmation of this prediction.

\paragraph{Weak recovery for the hypergraph stochastic block model}
Our second application concerns weak recovery in a two-community hypergraph stochastic block model. The special HSBM model, denoted $\HSBM(n,r,a,b)$, generates a random label $X_v\sim \Unif(\{\pm\})$ for each vertex and then places each $r$-uniform hyperedge with probability $\frac a{\binom n{r-1}}$ if all its labels agree and with probability $\frac b{\binom n{r-1}}$ otherwise.
Let
\begin{align} \label{eqn:hsbm-boht-params}
  d = \frac{a + \left(2^{r-1}-1\right)b}{2^{r-1}}, \qquad \lambda = \frac{a-b}{a+\left(2^{r-1}-1\right)b}.
\end{align}
The expected degree of any vertex is $d\pm o(1)$. The range of $\lambda$ is $\left[ -\frac 1{2^{r-1}-1}, 1 \right]$. The Kesten-Stigum (KS) threshold is $(r-1)d\lambda^2=1$.

The weak recovery problem for the HSBM asks whether one can recover from the hypergraph a vertex labeling that is positively correlated with the ground truth labeling $X$.
That is, we observe the hypergraph $G$ generated from the HSBM but not the community labels $X$, and the goal is to design an estimator $\wh X = \wh X(G)$ such that for some $\epsilon>0$,
\begin{align}
  \lim_{n\to \infty} \bP\left[d_H(X, \wh X) < \left( \frac 12-\epsilon\right)n\right] = 1,
\end{align}
where
\begin{align}
  d_H(X,\wh X) = \min_{s\in \{\pm\}} \sum_{v\in V} \mathbbm1\{X_v \ne s \wh X_v\}.
\end{align}
The labels are identifiable only up to a global sign flip.

\cite{pal2021community,stephan2022sparse} proved that weak recovery is possible above the KS threshold. The weak recovery threshold does not always match the KS threshold, but they are known to match in many cases (\cite{yuan2022testing,gu2023weak,gu2024community}): (1) $r=3,4$, $\lambda>0$; (2) $r\le 6$, $|\lambda|$ small enough; (3) any $r$, $\lambda \ge \frac 15$.

Using known relationships between BOHT and the hypergraph stochastic block model (HSBM) (e.g., \cite{gu2023channel}), \cref{thm:boht} expands the known parameter regimes in which KS is tight for weak recovery in the special HSBM.
\begin{theorem}[Weak recovery threshold for the special HSBM] \label{thm:hsbm}
  Consider the special HSBM model $\HSBM(n,r,a,b)$ where $r\in \bZ_{\ge 3}$, $a,b\in \bR_{\ge 0}$.
  Let $d$ and $\lambda$ be defined as in \cref{eqn:hsbm-boht-params}.
  If $(r-1) d \lambda^2 \le 1$ and one of the following is true, then weak recovery is impossible.
  \begin{enumerate}[label=(\roman*)]
    \item\label{item:thm:hsbm:i} $r=3$ or $r=4$;
    \item\label{item:thm:hsbm:ii} $r=5$ and $\lambda\ge -\frac 1{18}$;
    \item\label{item:thm:hsbm:iii} $r=6$ and $\lambda\ge \frac 1{46}$.
  \end{enumerate}
\end{theorem}

\cite{mossel2025weak} proved that the weak recovery threshold for the HSBM matches the detection threshold, the threshold at which the HSBM is indistinguishable from an Erd\H{o}s-R\'enyi hypergraph with the same average degree. Thus \cref{thm:hsbm} also establishes the detection threshold in these regimes.

\subsection{Our method} \label{sec:intro:method}
We next describe the method for proving \cref{thm:boht}. First we briefly review previous approaches to the reconstruction problem for BOHT.

\paragraph{Previous methods}
Two approaches have been effective in establishing tightness of KS in reconstruction problems on (hyper)trees.

The first proves contraction of information measures using strong data processing inequalities (SDPI) and subadditivity. \cite{borgs2006kesten} can be seen as an early example of this approach.
\cite{kulske2009symmetric,gu2023non} studied the Potts model using symmetric KL (SKL) information and mutual information respectively. \cite{gu2023weak} generalized this method to BOHT by introducing a multi-terminal version of SDPI, and proved tight non-reconstruction results using SKL information and $\chi^2$-information.
In this approach, one needs to find an information measure of channels that is subadditive under $\star$-product, and then prove that the SDPI contraction coefficient matches the KS threshold. The resulting non-reconstruction proofs are simple, but the choice of information measure is mysterious; when the contraction coefficient does not match the KS threshold, the method does not yield tight non-reconstruction results.

The second is Sly's method, which uses large degree asymptotics and robust non-reconstruction. This method was developed in \cite{sly2009reconstruction,sly2011reconstruction,mossel2023exact} for the random coloring model and the Potts model, and has been successful in other models such as the binary asymmetric model (\cite{liu2019large}) and BOHT (\cite{gu2024community}). It uses large degree asymptotics to prove that $\lim_{k\to \infty} C\left(\BP^k(\Id)\right)$ (for a certain information measure $C$) can be arbitrarily small when $d$ is large, and then uses a Taylor expansion to prove that for weak enough channels $P$, $C\left(\BP^k(P)\right)$ tends to $0$ as $k$ goes to infinity. This method is powerful and can establish very precise behavior of the evolution of $C\left(\BP^k(\Id)\right)$, thus giving tight non-reconstruction results. However, it requires a large degree assumption, and the Taylor expansion step requires heavy derivation.

\paragraph{Our method}
Our method combines an improved robust non-reconstruction analysis with a computer-assisted rigorous population dynamics algorithm, improving both parts of Sly's method and overcoming limitations of the previous approaches.

Let us state a relaxed version of Sly's framework in our setting. Recall our goal is to prove that
\begin{align}
  \lim_{k\to \infty} C\left(\BP^k(\Id)\right) = 0
\end{align}
for a certain information measure $C$. We use $C_{\chi^2}$, the $\chi^2$-capacity.
The proof strategy consists of two steps.
\begin{enumerate}
  \item We first prove a robust non-reconstruction result: for some robustness parameter $\epsilon>0$, for all BMS channels $P$ with $C_{\chi^2}(P) < \epsilon$, we have
  \begin{align}
    \lim_{k\to \infty} C_{\chi^2}\left(\BP^k(P)\right) = 0.
  \end{align}
  \item Then, we show that there exists a $k$ such that $C_{\chi^2}\left(\BP^k(\Id)\right) < \epsilon$.
\end{enumerate}
We improve both steps, leading to stronger non-reconstruction results.

\paragraph{Improved robust non-reconstruction}
In Sly's method, Step 1 is usually proved by a Taylor expansion of the BP operator.
Because the expansion is complicated, it can be unclear what values of $\epsilon$ it provides, which in turn makes Step 2 harder.

We improve the robust non-reconstruction analysis by exploiting a limitation of previous methods.
Recall the definition of the BP operator (\cref{eqn:intro:bp-operator}).
The SDPI-based method of \cite{gu2023weak} requires the information measure to contract at the KS rate, i.e.,
\begin{align} \label{eqn:intro:robust-recon-dream}
  C_{\chi^2}\left(P^{\times (r-1)} \circ B\right) \le (r-1) \lambda^2 C_{\chi^2}(P),
\end{align}
which is true in some but not all regimes.
For example, the Taylor expansion analysis (\cite{gu2024community}) shows that for small $C_{\chi^2}(P)$, we have
\begin{align} \label{sec:intro:method:bp-step-1}
  C_{\chi^2}\left(P^{\times (r-1)} \circ B\right) \le
  (r-1) \lambda^2 C_{\chi^2}(P) + c \lambda^2 \left(C_{\chi^2}(P)\right)^2
\end{align}
for some constant $c$ depending on $r$.
This $c$ may be positive, and in such cases \cref{eqn:intro:robust-recon-dream} cannot hold.

Our key observation is that the positive second-order term in \cref{sec:intro:method:bp-step-1} can be compensated for by a negative second-order term provided by the $\star$-power part of the BP operator.
Recall that the BP operator takes a mixture of $\star$-powers of these $P^{\times (r-1)} \circ B$ channels.
In the SDPI-based methods, the $\star$-powers are handled using subadditivity, i.e.,
\begin{align}
  C_{\chi^2}\left(P^{\star b}\right) \le b C_{\chi^2}(P).
\end{align}
This subadditivity is tight to the first order, but is quite lossy in the second order.
In fact, we have
\begin{align} \label{sec:intro:method:bp-step-2}
  C_{\chi^2}\left(P^{\star b}\right) \le 1-\left(1-C_{\chi^2}(P)\right)^b
  \le b C_{\chi^2}(P) - c' \binom b2 \left(C_{\chi^2}(P)\right)^2,
\end{align}
for some constant $c'>0$.
That is, $\star$-power gives us some advantage in the second order.
Combining \cref{sec:intro:method:bp-step-1,sec:intro:method:bp-step-2}, the second-order advantage from the $\star$-power part can outweigh the disadvantage in the contraction part. In certain regimes (such as $r=3$, any $\lambda$ or $r=5$, $\lambda>0$), the advantage is strong enough to directly imply impossibility of reconstruction.
In other regimes, it implies robust non-reconstruction with a useful robustness parameter $\epsilon$. For example, for $r=4$ and $5$, we can take $\epsilon=0.2$.

\paragraph{Rigorous population dynamics}
In Sly's method, Step 2 is proved using a Gaussian approximation, which requires a large degree assumption. We avoid this limitation with a rigorous computer program that bounds $C_{\chi^2}\left(\BP^k(\Id)\right)$.
This requires non-trivial effort, as the channel $\BP^k(\Id)$ can have countable support when the degree distribution $D$ has countable support. Even if $D=\mathbbm{1}_d$ is a point distribution, the support size of $\BP^k(\Id)$ can grow doubly exponentially in $k$, making it infeasible to maintain the exact channel on a realistic computer even for very small values such as $k=5$.
Physicists (e.g., \cite{abou1973selfconsistent,mezard2001bethe,mezard2006reconstruction}) developed a non-rigorous approximation called population dynamics for the channel $\BP^k(\Id)$.
In this method, a small-support approximation channel $P_k$ is computed for $\BP^k(\Id)$.
In each BP step (computing $P_{k+1}$ from $P_k$), one samples from the support of the channel $P_k$, applies the BP equations, and treats the resulting collection as an approximation of $P_{k+1}$.
Although this method can provide good estimates of the reconstruction threshold, it has limited theoretical guarantees because the sampling process is inherently noisy.

We develop a rigorous population dynamics algorithm to bound $\BP^k(\Id)$.
Instead of computing an approximation channel of $\BP^k(\Id)$, we compute a small-support upper bound $P_k$ of $\BP^k(\Id)$ in the sense of less-noisy preorder, an information-theoretic method for comparing information channels.
When computing $P_{k+1}$ from $P_k$, instead of taking a sample, we compute a small-support upper bound of $\BP(P_k)$ using known properties of the less-noisy preorder.
In this way, we can rigorously upper bound $C_{\chi^2}\left(\BP^k(\Id)\right)$ for finite values of $k$.

One final issue is parameter coverage.
We would like to prove non-reconstruction results for a continuous set of parameters $\lambda\in [\lambda_*,\lambda^*]$, but we are only able to do a finite amount of computation.
Our observation is that in Step 2, we do not need to confine the computation to points below the Kesten-Stigum threshold $(r-1) d \lambda^2 = 1$.
Slightly above the Kesten-Stigum threshold, reconstruction is known to be possible, but it may still be true that
\begin{align}
  \lim_{k\to \infty} C_{\chi^2}\left(\BP^k(\Id)\right)
\end{align}
is less than the robustness parameter $\epsilon$. Step 1 works only below the Kesten-Stigum threshold, so this observation alone does not prove non-reconstruction.
By a monotonicity argument, if
\begin{align}
  \lim_{k\to \infty} C_{\chi^2}\left(\BP_{r,d,\lambda}^k(\Id)\right) < \epsilon
\end{align}
for some $(r-1) d \lambda^2 > 1$, then
\begin{align}
  \lim_{k\to \infty} C_{\chi^2}\left(\BP_{r,d',\lambda'}^k(\Id)\right) < \epsilon
\end{align}
for all $(d',\lambda')$ where $d'\le d$, $\sgn(\lambda')=\sgn(\lambda)$, and $|\lambda'| \le |\lambda|$.
This allows a finite set of parameter computations to cover the whole interval, proving the desired result.

We remark that previously a similar rigorous population dynamics method has been used in \cite{gu2020broadcasting} for estimating bounds on the limit information in the Ising model on a tree.

\subsection{Organization} \label{sec:intro:org}
In \cref{sec:prelim}, we review preliminaries on information channels and broadcasting on (hyper)trees.
In \cref{sec:robust}, we present our improved robust non-reconstruction analysis.
In \cref{sec:pop}, we present our rigorous population dynamics algorithm.
In \cref{sec:special-boht}, we apply the methods developed in \cref{sec:robust,sec:pop} to the special BOHT model and prove \cref{thm:boht}.
In \cref{sec:app}, we apply \cref{thm:boht} to random NAE-SAT, random hypergraph bicoloring, and the special HSBM, proving \cref{thm:nae-sat,thm:hsbm}.
In \cref{sec:discuss}, we discuss further directions.

\section{Preliminaries} \label{sec:prelim}
\subsection{Information channels} \label{sec:prelim:bms}
We review basic definitions and facts about information channels, especially BMS channels. See \cite{richardson2008modern,polyanskiy2025information} for more background.

A channel (or information channel) $P: \cX \to \cY$ is a Markov kernel from $\cX$ to $\cY$.

Let $P: \cX \to \cY$, $Q: \cY \to \cZ$ be two channels. Their composition is denoted $Q\circ P: \cX \to \cZ$.

Let $P: \cX \to \cY$, $Q: \cX \to \cZ$ be two channels with the same input alphabet. Their $\star$-product is the channel $P\star Q: \cX \to \cY\times \cZ$, such that for $x\in \cX$ and measurable subsets $E\subseteq \cY$, $F\subseteq \cZ$ we have $(P\star Q)(E\times F | x) = P(E|x) Q(F | x)$.
We use $P^{\star n}$ to denote the $n$-th ($n\ge 0$) $\star$-power of $P$.
When $n=0$, $P^{\star 0}$ is the trivial channel $\cX \to \{*\}$.

Let $P: \cX \to \cZ$, $Q: \cY \to \cW$ be two channels. Their product is the channel $P\times Q: \cX \times \cY \to \cZ \times \cW$, such that for $x\in \cX$, $y\in \cY$ and measurable subsets $E\subseteq \cZ$, $F\subseteq \cW$ we have $(P\times Q)(E\times F | x,y) = P(E|x) Q(F|y)$.
We use $P^{\times n}$ to denote the $n$-th ($n\ge 0$) tensor power of $P$.
When $n=0$, $P^{\times 0}$ is the trivial channel $\{*\} \to \{*\}$.

Let $\cA = \{P_I: \cX \to \cY_I | I\in \cI\}$ be a collection of information channels with the same input alphabet and $\mu$ be a distribution on $\cI$. Then we define a channel $\bE_{I\sim \mu} P_I: \cX \to \bigsqcup_{I\in \cI} \cY_I$ that maps $x\in \cX$ to $(I, Y)$, where $I\sim \mu$ and $Y\sim P_I(\cdot | x)$.
This is called a mixture of $\cA$.

Let $P: \cX \to \cY$, $Q: \cX \to \cZ$ be two channels with the same input alphabet. We say $P$ is more degraded than $Q$, denoted $P \le_{\deg} Q$, if there exists a channel $R: \cZ \to \cY$ such that $P=R\circ Q$.
In other words, $P\le_{\deg} Q$ if and only if we can simulate $P$ given access to $Q$.
We say $P$ and $Q$ are equivalent if $P \le_{\deg} Q$ and $Q \le_{\deg} P$.

A binary memoryless symmetric (BMS) channel is a binary-input channel $P: \{\pm\} \to \cY$ together with a measurable involution $\sigma: \cY \to \cY$ such that for every measurable subset $E\subseteq \cY$, we have $P(E|+) = P(\sigma(E)|-)$. The involution is often evident and omitted.

Examples of BMS channels include the BSC channel $\BSC_\delta: \{\pm\} \to \{\pm\}$, $\delta\in [0, 1]$, defined by
\begin{align}
  \BSC_\delta(y|x) = \left\{
    \begin{array}{ll}
      1-\delta, & y=x, \\
      \delta, & y=-x,
    \end{array}
  \right.
\end{align}
and the BEC channel $\BEC_\epsilon: \{\pm\} \to \{\pm,*\}$, $\epsilon\in [0, 1]$, defined by
\begin{align}
  \BEC_\epsilon(y|x) = \left\{
    \begin{array}{ll}
      1-\epsilon, & y=x, \\
      0, & y=-x, \\
      \epsilon, & y=*.
    \end{array}
  \right.
\end{align}
Note that $\BEC_\epsilon$ is equivalent to a mixture of BSCs. That is, $\BEC_\epsilon = \epsilon \BSC_{1/2} + (1-\epsilon) \BSC_0$.

In general, BMS channels are equivalent to mixtures of BSCs.
That is, for any BMS channel $P$, there is a unique distribution $\mu_P$ on the interval $\left[0, \frac 12\right]$ such that $P = \bE_{\Delta_P\sim \mu_P} \BSC_{\Delta_P}$.
Therefore, equivalence classes of BMS channels are in one-to-one correspondence with distributions on the interval $\left[0, \frac 12\right]$.
We say $\Delta_P$ is the $\Delta$-component of $P$.
The $\theta$-component of $P$ is a random variable $\theta_P$ on $[0,1]$ satisfying $\theta_P = 1-2\Delta_P$.

For an $f$-divergence, a channel $P: \cX \to \cY$ and a distribution $\mu$ on $\cX$, we write $I_f(\mu, P)$ to be the $f$-mutual information between $X$ and $Y$, where $X\sim \mu$ and $Y\sim P(\cdot | X)$.

The following $f$-divergences are useful. Let $\mu,\nu$ be two distributions on the same measurable space $\cX$.
\begin{itemize}
  \item Total variation (TV) distance
  \begin{align}
    \TV(\mu, \nu) = \sup_{\cA\subseteq \cX} \left| \mu(\cA) - \nu(\cA) \right|,
  \end{align}
  where $\cA$ ranges over measurable subsets of $\cX$.
  TV distance is an $f$-divergence with $f(x) = \frac 12 |x-1|$.
  \item Kullback-Leibler (KL) divergence
  \begin{align}
    D(\mu \| \nu) = \int \frac{d\mu}{d\nu}\log \frac{d\mu}{d\nu} d\nu.
  \end{align}
  \item $\chi^2$-divergence
  \begin{align}
    \chi^2(\mu \| \nu) = \int \left( \frac{d\mu}{d\nu}-1 \right)^2 d\nu.
  \end{align}
\end{itemize}

Let $\pi = \Unif(\{\pm\})$ and $P$ be a BMS channel. Let $\Delta_P$ be the $\Delta$-component of $P$ and $\theta_P$ be the $\theta$-component of $P$. We define the following information measures of $P$.
\begin{itemize}
  \item Capacity
  \begin{align}
    C(P) = I(\pi, P) = \bE \left[\log 2 + \Delta_P \log \Delta_P + (1-\Delta_P) \log (1-\Delta_P) \right].
  \end{align}
  \item $\chi^2$-capacity
  \begin{align}
    C_{\chi^2}(P) = I_{\chi^2}(\pi, P) = \bE (1-2\Delta_P)^2 = \bE \theta_P^2.
  \end{align}
\end{itemize}

Let $\mu$, $\nu$ be two distributions on $\bR$. We say $\mu$ second-order stochastically dominates $\nu$ if there is a coupling $(X,Y)$ of $\mu$ and $\nu$ such that $\bE[X|Y] \ge Y$ almost surely.

The following lemma shows that degradation between BMS channels can be stated using second-order stochastic dominance between their $\Delta$-variables. A proof can be found in e.g., \cite[Theorem 4.74]{richardson2008modern}.
\begin{lemma} \label{lem:bms-deg-coupling}
  Let $P$ and $Q$ be two BMS channels.
  Then $Q\le_{\deg} P$ if and only if $\theta_P$ second-order stochastically dominates $\theta_Q$.
\end{lemma}

\subsection{Less-noisy preorder} \label{sec:prelim:less-noisy}
Degradation is a useful channel preorder, but it is sometimes too strong. For example, for the channel $\BSC_\delta$, the largest $\epsilon$ satisfying $\BSC_\delta \le_{\deg} \BEC_\epsilon$ is $\epsilon=2\delta$. However, $\BSC_\delta$ contracts mutual information at least as strongly as $\BEC_{1-(1-2\delta)^2}$, which is a strictly weaker channel than $\BEC_{2\delta}$ when $0<\delta<\frac 12$.
\cite{korner1977comparison} defined the less-noisy preorder, a weaker preorder than degradation that is more suitable for certain problems.
\begin{definition}[Less-noisy preorder] \label{defn:less-noisy}
  For two channels $P: \cX \to \cY$, $Q: \cX\to \cZ$, we say $P$ is less noisy than $Q$, denoted $Q \le_{\ln} P$, if for all channels $R: \cW \to \cX$ and all distributions $\mu$ on $\cW$, we have
  \begin{align} \label{eqn:defn:less-noisy-kl}
    I(\mu, Q\circ R) \le I(\mu, P\circ R).
  \end{align}
\end{definition}
The less-noisy preorder is not specific to mutual information. \cite[Theorem 1]{makur2018comparison} showed that replacing \cref{eqn:defn:less-noisy-kl} with
\begin{align} \label{eqn:defn:less-noisy-chi2}
  I_{\chi^2}(\mu, Q\circ R) \le I_{\chi^2}(\mu, P\circ R)
\end{align}
gives the same preorder.
In fact, the less-noisy preorder can be defined using any non-linear operator convex $f$-divergence (\cite[Theorem 3.1]{makur2019information}).

Both the degradation preorder and the less-noisy preorder respect pre-composition, $\star$-product, product, and mixture, summarized as follows. The proof that the less-noisy preorder respects product can be found in e.g., \cite[Prop.~16]{polyanskiy2017strong}.
\begin{itemize}
  \item (Pre-composition) Let $P: \cX \to \cY$, $Q: \cX \to \cZ$ be two channels. If $P \le_{\deg} Q$, then $P\circ R\le_{\deg} Q\circ R$ for any $R: \cW\to \cX$. If $P \le_{\ln} Q$, then $P\circ R\le_{\ln} Q\circ R$ for any $R: \cW\to \cX$.
  \item ($\star$-product) Let $P_1: \cX \to \cY_1$, $P_2: \cX \to \cY_2$, $Q_1: \cX \to \cZ_1$, $Q_2: \cX \to \cZ_2$ be four channels with the same input alphabet. If $P_1 \le_{\deg} Q_1$ and $P_2 \le_{\deg} Q_2$, then $P_1 \star P_2 \le_{\deg} Q_1 \star Q_2$. If $P_1 \le_{\ln} Q_1$ and $P_2 \le_{\ln} Q_2$, then $P_1 \star P_2 \le_{\ln} Q_1 \star Q_2$.
  \item (Product) Let $P_1: \cX_1 \to \cY_1$, $P_2: \cX_2 \to \cY_2$, $Q_1: \cX_1 \to \cZ_1$, $Q_2: \cX_2 \to \cZ_2$ be four channels. If $P_1 \le_{\deg} Q_1$ and $P_2 \le_{\deg} Q_2$, then $P_1 \times P_2 \le_{\deg} Q_1 \times Q_2$. If $P_1 \le_{\ln} Q_1$ and $P_2 \le_{\ln} Q_2$, then $P_1 \times P_2 \le_{\ln} Q_1 \times Q_2$.
  \item (Mixture) Let $\cA = \{P_I: \cX \to \cY_I | I\in \cI\}$ and $\cB = \{Q_I: \cX \to \cZ_I | I\in \cI\}$ be two collections of channels with the same index set $\cI$. Let $\mu$ be a distribution on $\cI$. If $P_I\le_{\deg} Q_I$ for all $I\in \cI$, then $\bE_{I\sim \mu} P_I \le_{\deg} \bE_{I\sim \mu} Q_I$. If $P_I\le_{\ln} Q_I$ for all $I\in \cI$, then $\bE_{I\sim \mu} P_I \le_{\ln} \bE_{I\sim \mu} Q_I$.
\end{itemize}

The following lemma generalizes one side of \cref{lem:bms-deg-coupling} to the less-noisy preorder.
A proof can be found in the proof of \cite[Lemma 2]{roozbehani2019low}.
We include a shorter computational proof.
\begin{lemma} \label{lem:bms-ln-coupling}
  Let $P$ and $Q$ be two BMS channels.
  If $\theta_P^2$ second-order stochastically dominates $\theta_Q^2$, then $Q\le_{\ln} P$.
\end{lemma}
\begin{proof}
  By \cite[Theorem 1]{makur2018comparison}, $Q \le_{\ln} P$ is equivalent to
  \begin{align}
    \chi^2(\mu Q \| \nu Q) \le \chi^2(\mu P \| \nu P)
  \end{align}
  for all distributions $\mu,\nu$ on $\{\pm\}$.
  Let $\mu = (x,1-x)$, $\nu = (y,1-y)$. Then
  \begin{align}
    \chi^2(\mu P \| \nu P) &= \bE_{\Delta_P} \chi^2(\mu \BSC_{\Delta_P} \| \nu \BSC_{\Delta_P}) \\
    \nonumber &= \bE_{\Delta_P} \left[\frac{4 \theta_P^2 (x-y)^2}{1-(1-2y)^2\theta_P^2}\right].
  \end{align}
  Therefore, to prove that $Q\le_{\ln} P$, it suffices to prove that the function
  \begin{align}
    f_y(t) = \frac{t}{1-(1-2y)^2 t}
  \end{align}
  is convex on $[0, 1]$ for any $0\le y\le 1$.
  Direct computation shows that
  \begin{align}
    f''_y(t) = \frac{2(1-2y)^2}{(1-(1-2y)^2 t)^3} \ge 0.
  \end{align}
  This finishes the proof.
\end{proof}

The following is a direct corollary of \cref{lem:bms-ln-coupling}, and can be found in \cite[Lemma 2]{roozbehani2019low}.
\begin{corollary} \label{lem:bms-ln-bec}
  For any BMS channel $P$ we have
  \begin{align}
    \BSC_{\frac 12 \left(1-\sqrt{C_{\chi^2}(P)}\right)} \le_{\ln} P\le_{\ln} \BEC_{1-C_{\chi^2}(P)}.
  \end{align}
\end{corollary}

\subsection{Broadcasting on hypertrees} \label{sec:prelim:boht}
\begin{definition}[Broadcasting on hypertrees] \label{defn:boht}
  Let $r\ge 2$ be an integer.
  Let $\cX$ be a finite set with $|\cX|\ge 2$.
  Let $\pi$ be a distribution on $\cX$ with full support.
  Let $B: \cX \to \cX^{r-1}$ be a channel satisfying
  \begin{align}
    B(y_1,\ldots,y_{r-1} | x) = B(y_{\tau(1)},\ldots,y_{\tau(r-1)} | x), \qquad \forall x,y_1,\ldots,y_{r-1}\in \cX, \tau\in S_{r-1}
  \end{align}
  where $S_{r-1}$ is the symmetric group on $r-1$ elements.
  Let $\wt B: \cX \to \cX$ be the channel
  \begin{align} \label{eqn:defn:boht:wtB}
    \wt B(y|x) = \sum_{y_2,\ldots,y_{r-1}\in \cX} B(y,y_2,\ldots,y_{r-1} | x), \qquad \forall x,y\in \cX.
  \end{align}
  We further assume that $\pi \wt B = \pi$.
  Let $T$ be a (possibly infinite) $r$-uniform linear hypertree rooted at $\rho$.

  The broadcasting on hypertrees (BOHT) model, denoted $\BOHT(T,r,\cX,\pi,B)$, generates a labeling $\sigma \in \cX^{V(T)}$ according to the following process.
  \begin{enumerate}[label=(\arabic*)]
    \item Generate $\sigma_\rho \sim \pi$.
    \item Suppose we have generated $\sigma_u$. Independently for every downward hyperedge $(u,v_1,\ldots,v_{r-1})$, generate $(\sigma_{v_1},\ldots,\sigma_{v_{r-1}}) \sim B(\cdot | \sigma_u)$.
  \end{enumerate}
  The output of the model is $(\sigma,T)$.

  Let $D$ be a distribution supported on non-negative integers.
  If $T$ is an $r$-uniform Galton-Watson linear hypertree with offspring distribution $D$ (i.e., the number of downward hyperedges of a vertex is sampled from $D$), we denote the model as $\BOHT(D,r,\cX,\pi,B)$.
\end{definition}

\begin{definition}[Binary symmetric BOHT models] \label{defn:bms-boht}
  A model $\BOHT(T,r,\cX,\pi,B)$ or $\BOHT(D,r,\cX,\pi,B)$ is called binary symmetric if $\cX=\{\pm\}$, $\pi = \Unif(\cX)$, and $B: \{\pm\} \to \{\pm\}^{r-1}$ (together with the global sign flip $\{\pm\}^{r-1} \to \{\pm\}^{r-1}$) is a BMS channel.
  In this case, the transition probability $B(y_1,\ldots,y_{r-1} | x)$ depends only on $\#\{i\in [r-1]: y_i=x\}$.
  Therefore the channel $B$ can be represented by a sequence $\bfb = (b_0,\ldots,b_{r-1})$ where $b_k = B\left( +^k -^{r-1-k} | +\right)$ ($0\le k\le r-1$).
  We call $(b_0,\ldots,b_{r-1})$ the signature of the model.
  Any sequence $\bfb=(b_0,\ldots,b_{r-1})$ satisfying $b_k\ge 0$ ($0\le k\le r-1$) and
  \begin{align}
    \sum_{0\le k\le r-1} \binom {r-1}k b_k =1
  \end{align}
  is the signature of a binary symmetric BOHT model, with the broadcasting channel $B=B_{\bfb}: \{\pm\} \to \{\pm\}^{r-1}$ given by
  \begin{align} \label{eqn:defn:bms-boht:Bb}
    B_{\bfb}(y_1,\ldots,y_{r-1} | x) = b_{\# \{i\in [r-1]: y_i=x\}}.
  \end{align}
  We denote a binary symmetric BOHT model with signature $\bfb$ as $\BOHT(T,r,\bfb)$ or $\BOHT(D,r,\bfb)$.
\end{definition}

\begin{definition}[Special BOHT model] \label{defn:special-boht}
  For $\lambda\in \left[-\frac 1{2^{r-1}-1}, 1\right]$, the signature
  \begin{align}
    \bfb = \left(\frac{1-\lambda}{2^{r-1}},\ldots,\frac{1-\lambda}{2^{r-1}},\lambda + \frac{1-\lambda}{2^{r-1}}\right)
  \end{align}
  gives rise to a binary symmetric BOHT model, which we call the special BOHT model.
  The broadcasting channel $B=B_{r,\lambda}: \{\pm\} \to \{\pm\}^{r-1}$ is given by
  \begin{align}
    B_{r,\lambda}(y_1,\ldots,y_{r-1} | x) = \left\{
      \begin{array}{ll}
        \lambda + \frac{1-\lambda}{2^{r-1}} & \text{if}~y_1=\cdots=y_{r-1}=x,\\
        \frac{1-\lambda}{2^{r-1}} & \text{o.w.}
      \end{array}
    \right.
  \end{align}
  We denote the special BOHT model as $\BOHT(T,r,\lambda)$ or $\BOHT(D,r,\lambda)$.
\end{definition}

\begin{definition}[Reconstruction problem for BOHT] \label{defn:boht-recon}
  Consider the model $\BOHT(T,r,\cX,\pi,B)$ or $\BOHT(D,r,\cX,\pi,B)$.
  For $k\ge 0$, define $L_k$ as the set of vertices at distance $k$ to $\rho$, and $T_k$ as the subgraph induced by $L_0\cup \cdots \cup L_k$.
  We say reconstruction is possible if any of the following equivalent conditions hold.
  \begin{enumerate}[label=(\alph*)]
    \item For some $x,y\in \cX$,
    \begin{align}
      \lim_{k\to \infty} \TV\left(P_{\sigma_{L_k} | \sigma_\rho=x}, P_{\sigma_{L_k} | \sigma_\rho=y} \mid T_k \right) > 0.
    \end{align}
    \item
    \begin{align}
      \lim_{k\to \infty} I(\sigma_\rho; T_k, \sigma_{L_k}) > 0.
    \end{align}
    \item
    \begin{align}
      \lim_{k\to \infty} I_{\chi^2}(\sigma_\rho; T_k, \sigma_{L_k}) > 0.
    \end{align}
  \end{enumerate}
  We say reconstruction is impossible if the conditions do not hold.
\end{definition}

The reconstruction problem on the Galton-Watson BOHT model $\BOHT(D,r,\cX,\pi,B)$ can be studied using the belief propagation operator.
\begin{definition}[Belief propagation operator] \label{defn:bp-operator}
  Consider the model $\BOHT(D,r,\cX,\pi,B)$.
  The belief propagation (BP) operator $\BP$ maps a channel $P$ with input alphabet $\cX$ to
  \begin{align}
    \BP(P) = \bE_{b\sim D} \left(P^{\times (r-1)} \circ B\right)^{\star b}.
  \end{align}
\end{definition}
The $\BP$ operator maps the space of channels with input alphabet $\cX$ to itself.
If the BOHT model is binary symmetric, $\BP$ maps BMS channels to BMS channels.

For $k\ge 0$, let $M_k$ denote the channel $\sigma_\rho \mapsto (T_k, \sigma_{L_k})$.
Then $M_k=\BP^k(\Id)$ where $\Id: \cX \to \cX$ is the identity channel.
Therefore we have the following equivalent conditions for reconstruction.
\begin{enumerate}[label=(\alph*)]
  \setItemnumber{4}
  \item
  \begin{align}
    \lim_{k\to \infty} I\left(\pi, \BP^k(\Id)\right) > 0.
  \end{align}
  \item
  \begin{align}
    \lim_{k\to \infty} I_{\chi^2}\left(\pi, \BP^k(\Id)\right) > 0.
  \end{align}
\end{enumerate}

\begin{definition}[Kesten-Stigum threshold] \label{defn:boht-ks}
  Consider the model $\BOHT(D,r,\cX,\pi,B)$.
  Recall the channel $\wt B: \cX \to \cX$ defined in \cref{eqn:defn:boht:wtB}.
  Let $d = \bE_{b\sim D} b$ and $\lambda = \lambda_2(\wt B)$ be the second largest eigenvalue (in absolute value) of $\wt B$.
  The Kesten-Stigum threshold is defined as $(r-1) d |\lambda|^2 = 1$.
\end{definition}
One can also define the Kesten-Stigum threshold for the model $\BOHT(T,r,\cX,\pi,B)$.
For an infinite rooted (hyper)tree $T$, the branching number $\br(T)$ generalizes the expected offspring $d$ for Galton-Watson (hyper)trees. The Kesten-Stigum threshold for $\BOHT(T,r,\cX,\pi,B)$ is $\br(T) |\lambda|^2 = 1$.
In this paper we focus on Galton-Watson (hyper)trees.

The Kesten-Stigum threshold is useful because it provides a simple criterion for reconstruction. A folklore result states that reconstruction is always possible above the Kesten-Stigum threshold (i.e., when $(r-1) d |\lambda|^2 > 1$).
In general, the reconstruction threshold can be strictly below the Kesten-Stigum threshold, as shown in e.g., \cite{mossel2001reconstruction} for the Potts model.

\begin{definition}[Robust reconstruction] \label{defn:boht-robust-recon}
  Consider the model $\BOHT(T,r,\cX,\pi,B)$ or $\BOHT(D,r,\cX,\pi,B)$.
  Let $W$ be a channel with input alphabet $\cX$.
  We say robust reconstruction is possible with respect to $W$ if
  \begin{align}
    \lim_{k\to \infty} I\left(\pi, \BP^k(W)\right) > 0.
  \end{align}
  We say robust reconstruction is impossible with respect to $W$ if the above condition does not hold.
\end{definition}

The next calculation gives the Kesten-Stigum threshold for binary symmetric BOHT models.
Consider $\BOHT(D,r,B_{\bfb})$. Note that the channel $\wt B_{\bfb}: \{\pm\}\to \{\pm\}$ is a BSC channel.
We have
\begin{align}
  \wt B_{\bfb}(+ | +)
  = \sum_{1\le k\le r-1} \binom{r-2}{k-1} B_{\bfb}\left(+^k -^{r-1-k} | +\right)
  = \sum_{1\le k\le r-1} b_k \binom{r-2}{k-1}.
\end{align}
So $\wt B_{\bfb} = \BSC_\delta$ for
\begin{align}
  \delta = 1-\sum_{1\le k\le r-1} b_k \binom{r-2}{k-1}.
\end{align}
Therefore
\begin{align}
  \lambda = 1-2\delta = 2\sum_{1\le k\le r-1} b_k \binom{r-2}{k-1}-1.
\end{align}

The following lemma relates BOHT models with different parameters.
\begin{lemma}[Monotonicity] \label{lem:boht-monotone}
  Consider two binary symmetric BOHT models $\BOHT(\Pois(d),r,\bfb)$ and $\BOHT(\Pois(d'),r,\bfb')$ satisfying $d\ge d'$ and
  \begin{align}
    \bfb' = (1-c) \bfb + c \left(\frac 1{2^{r-1}}, \ldots, \frac 1{2^{r-1}}\right)
  \end{align}
  for some $c\in [0, 1]$.
  Then for any BMS channel $P$ and any $k\ge 0$ we have
  \begin{align}
    \BP_{r,d,\bfb}^k(P) \ge_{\deg} \BP_{r,d',\bfb'}^k(P).
  \end{align}

  In particular, for special BOHT models $\BOHT(\Pois(d),r,\lambda)$ and $\BOHT(\Pois(d'),r,\lambda')$,
  if $d\ge d'$, $\sgn(\lambda)=\sgn(\lambda')$ and $|\lambda|\ge|\lambda'|$, then for any BMS channel $P$ and any $k\ge 0$ we have
  \begin{align}
    \BP_{r,d,\lambda}^k(P) \ge_{\deg} \BP_{r,d',\lambda'}^k(P).
  \end{align}
\end{lemma}
\begin{proof}
  Because the degradation preorder respects the BP operator, it suffices to prove the case $k=1$.
  Because there is a monotone coupling between $\Pois(d)$ and $\Pois(d')$, it suffices to prove the case $d=d'$.
  Now
  \begin{align}
    \BP_{r,d,\bfb}(P) = \bE_{b\sim \Pois(d)} \left(P^{\times (r-1)} \circ B_{r,\bfb}\right)^{\star b},
    \qquad
    \BP_{r,d,\bfb'}(P) = \bE_{b\sim \Pois(d)} \left(P^{\times (r-1)} \circ B_{r,\bfb'}\right)^{\star b}.
  \end{align}
  So it suffices to prove that
  \begin{align} \label{eqn:lem:boht-monotone:step}
    P^{\times (r-1)} \circ B_{r,\bfb} \ge_{\deg} P^{\times (r-1)} \circ B_{r,\bfb'}.
  \end{align}
  Let $\cY$ be the output alphabet of $P$.
  Then the output alphabets of $P^{\times (r-1)} \circ B_{r,\bfb}$ and $P^{\times (r-1)} \circ B_{r,\bfb'}$ are both $\cY^{r-1}$.
  Let $R: \cY^{r-1} \to \cY^{r-1}$ be the channel that, on input $y^{r-1}\in \cY^{r-1}$, outputs $y^{r-1}$ with probability $1-c$ and outputs a sample from $(P\circ \Unif(\{\pm\}))^{\times (r-1)}$ with probability $c$.
  Then we have
  \begin{align}
    R\circ P^{\times (r-1)} \circ B_{r,\bfb} = P^{\times (r-1)} \circ B_{r,\bfb'},
  \end{align}
  which implies \cref{eqn:lem:boht-monotone:step}.
\end{proof}

\section{Robust non-reconstruction} \label{sec:robust}
In this section, we use the less-noisy preorder to provide an upper bound for $\BP(P)$ (\cref{prop:robust}).
A Taylor expansion then gives criteria for robust non-reconstruction (\cref{coro:robust}). While our main application (\cref{thm:boht}) is the special BOHT model, we state the method for binary symmetric BOHT models.

\begin{proposition} \label{prop:robust}
  Consider a binary symmetric BOHT model $\BOHT(D,r,\bfb)$.
  Let $\BP$ be the associated belief propagation operator.
  For any BMS channel $P$, we have
  \begin{align} \label{eqn:prop:robust:final}
    \BP(P) \le_{\ln} \BEC_{1-f\left(C_{\chi^2}(P)\right)},
  \end{align}
  where $f: [0, 1]\to [0, 1]$ is defined as
  \begin{align}
    \label{eqn:prop:robust:f} f(x) &= 1-\bE_{b\sim D} (1-g(x))^b,\\
    \label{eqn:prop:robust:g} g(x) &= \sum_{0\le i\le r-1} \binom{r-1}i x^i (1-x)^{r-1-i} c_i,\\
    \label{eqn:prop:robust:c} c_i &= \sum_{0\le j\le i} p_{i,j} \left(\frac{p_{i,j}-p_{i,i-j}}{p_{i,j}+p_{i,i-j}}\right)^2, \qquad 0\le i\le r-1,\\
    \label{eqn:prop:robust:p} p_{i,j} &= \sum_{0\le k\le r-1} b_k \binom ij \binom {r-1-i}{k-j} , \qquad 0\le j\le i\le r-1.
  \end{align}
\end{proposition}
\begin{proof}
  Let $x = C_{\chi^2}(P)$.
  Recall that
  \begin{align}
    \BP(P) = \bE_{b\sim D} \left( P^{\times (r-1)} \circ B_{\bfb}\right)^{\star b}
  \end{align}
  where $B_{\bfb}$ is defined in \cref{eqn:defn:bms-boht:Bb}.

  We first prove that
  \begin{align} \label{eqn:proof:prop:robust:step0}
    P^{\times (r-1)} \circ B_{\bfb} \le_{\ln} \BEC_{1-g(x)}.
  \end{align}
  By \cref{lem:bms-ln-bec}, it suffices to prove that
  \begin{align} \label{eqn:proof:prop:robust:step1}
    C_{\chi^2}\left(\BEC_{1-x}^{\times (r-1)} \circ B_{\bfb}\right) \le g(x).
  \end{align}
  We have
  \begin{align} \label{eqn:proof:prop:robust:step2}
    &~C_{\chi^2}\left(\BEC_{1-x}^{\times (r-1)} \circ B_{\bfb}\right)\\
    \nonumber =& C_{\chi^2}\left(\bE_{z_1,\ldots,z_{r-1} \sim \Ber(x)} \left(\left(\prod_{i\in [r-1]} \BEC_{1-z_i} \right)\circ B_{\bfb}\right)\right) \\
    \nonumber =& \bE_{z_1,\ldots,z_{r-1} \sim \Ber(x)} C_{\chi^2}\left(\left(\prod_{i\in [r-1]} \BEC_{1-z_i} \right)\circ B_{\bfb}\right) \\
    \nonumber =& \sum_{0\le i\le r-1} \binom{r-1}i x^i (1-x)^{r-1-i} C_{\chi^2}\left(\left(\Id^{\times i} \times 0^{\times (r-1-i)}\right)\circ B_{\bfb}\right)
  \end{align}
  where $\Id=\BEC_0$ is (equivalent to) the identity channel and $0=\BEC_1$ is (equivalent to) the trivial channel.
  Comparing \cref{eqn:prop:robust:g,eqn:proof:prop:robust:step2}, to prove \cref{eqn:proof:prop:robust:step1}, we only need to prove that
  \begin{align} \label{eqn:proof:prop:robust:step3}
    C_{\chi^2}\left(\left(\Id^{\times i} \times 0^{\times (r-1-i)}\right)\circ B_{\bfb}\right) = c_i.
  \end{align}
  Let $B_i$ denote the channel $\left(\Id^{\times i} \times 0^{\times (r-1-i)}\right)\circ B_{\bfb}$.
  Let $R: \{\pm,*\}^{r-1} \to \{0,\ldots,r-1\}$ be the channel that maps $(y_1,\ldots,y_{r-1})$ to $\#\{j\in [r-1]: y_j=+\}$.
  Then $R\circ B_i$ is equivalent to $B_i$.
  For $0\le j\le i$, we have
  \begin{align}
    (R\circ B_i)(j | +) = \sum_{0\le k\le r-1} b_k \binom ij \binom{r-1-i}{k-j} = p_{i,j}.
  \end{align}
  Note that $R\circ B_i$ together with the map $\sigma(j)=i-j$ is a BMS channel.
  Therefore
  \begin{align}
    C_{\chi^2}(B_i) = C_{\chi^2}(R\circ B_i) = \sum_{0\le j\le i} p_{i,j} \left(\frac{p_{i,j}-p_{i,i-j}}{p_{i,j}+p_{i,i-j}}\right)^2 = c_i.
  \end{align}
  This proves \cref{eqn:proof:prop:robust:step3} and thus \cref{eqn:proof:prop:robust:step0}.

  By \cref{eqn:proof:prop:robust:step0},
  \begin{align}
    C_{\chi^2}(\BP(P)) &= C_{\chi^2}\left(\bE_{b\sim D} \left( P^{\times (r-1)} \circ B_{\bfb}\right)^{\star b}\right) \\
    \nonumber &\le C_{\chi^2}\left(\bE_{b\sim D} \BEC_{1-g(x)}^{\star b}\right) \\
    \nonumber &= \bE_{b\sim D} \left[1-(1-g(x))^b\right] \\
    \nonumber &= f(x).
  \end{align}
  This implies \cref{eqn:prop:robust:final} by \cref{lem:bms-ln-bec}.
\end{proof}
We remark that $c_0$ is always zero. Therefore in \cref{eqn:prop:robust:g}, we can take the summation from $1$ to $r-1$.
\begin{corollary} \label{coro:robust}
  Work in the setting of \cref{prop:robust}.
  \begin{enumerate}[label=(\roman*)]
    \item If $(r-1) d \lambda^2 < 1$, then there exists $\epsilon>0$ such that robust reconstruction is impossible with respect to all BMS channels $W$ with $C_{\chi^2}(W) \le \epsilon$.
    \item If $(r-1) d \lambda^2 = 1$ and
    \begin{align}
      \binom{r-1}2 d c_2 < r-2 + (r-1)^2 \left(\bE_{b\sim D} \binom b2\right) \lambda^4,
    \end{align}
    then there exists $\epsilon>0$ such that robust reconstruction is impossible with respect to all BMS channels $W$ with $C_{\chi^2}(W) \le \epsilon$.
  \end{enumerate}
\end{corollary}
\begin{proof}
  By \cref{eqn:prop:robust:g},
  \begin{align}
    g(x) &= (r-1) x(1-(r-2)x) c_1 + \binom{r-1}2 x^2 c_2 + O_{\bfb}(x^3) \\
    \nonumber &= (r-1) c_1 x + \binom{r-1}2 (c_2 - 2 c_1) x^2 + O_{\bfb}(x^3).
  \end{align}
  Then by \cref{eqn:prop:robust:f},
  \begin{align}
    f(x) &= 1-\bE_{b\sim D} \left[ 1-b g(x) + \binom b2 g(x)^2 \right] + O_D(x^3), \\
    \nonumber &= d g(x) - \left(\bE_{b\sim D} \binom b2\right) g(x)^2 + O_D(x^3) \\
    \nonumber &= (r-1) d c_1 x + \left( \binom{r-1}2 d (c_2-2c_1) - (r-1)^2 \left(\bE_{b\sim D} \binom b2\right) c_1^2 \right) x^2 + O_{D,\bfb}(x^3).
  \end{align}
  Therefore, if $(r-1) d c_1 < 1$ or $(r-1) d c_1 = 1$ and
  \begin{align}
    \binom{r-1}2 d (c_2-2c_1) - (r-1)^2 \left(\bE_{b\sim D} \binom b2\right) c_1^2 < 0,
  \end{align}
  then there exists $\epsilon>0$ such that $f(x) < x$ for all $0\le x\le \epsilon$.

  It remains to prove that $c_1 = \lambda^2$.
  We have
  \begin{align}
    p_{1,1}-p_{1,0} &= \sum_{0\le k\le r-1} b_k \binom{r-2}{k-1} - \sum_{0\le k\le r-1} b_k \binom{r-2}{k} \\
    \nonumber &= \sum_{0\le k\le r-1} b_k \left(2 \binom{r-2}{k-1} - \binom{r-1}k\right) \\
    \nonumber &= 2 \sum_{0\le k\le r-1} b_k \binom{r-2}{k-1} - 1\\
    \nonumber &= \lambda
  \end{align}
  and $p_{1,1}+p_{1,0}=1$. So
  \begin{align}
    c_1 = (p_{1,1}-p_{1,0})^2 = \lambda^2.
  \end{align}
\end{proof}

\section{Rigorous population dynamics} \label{sec:pop}
In this section we present our rigorous population dynamics algorithm (\cref{alg:pop,alg:bms-builder}), which produces an upper bound for $\BP^k(\Id)$. Our algorithm uses finite-precision arithmetic and can be implemented using integers of bounded size. In \cref{sec:special-boht}, we use a rigorous implementation of the algorithm to prove non-reconstruction results for the special BOHT\@.
While our main application (\cref{thm:boht}) is the special BOHT model, we state the method for binary symmetric BOHT models.

\begin{algorithm}[ht]
  \caption{Rigorous population dynamics}
  \label{alg:pop}
  \begin{algorithmic}[1]
    \State \textbf{global parameters:} Support size $s\in \bZ_{\ge 1}$. Precision parameter $w \in \bZ_{\ge 1}$.
    \Procedure{PopulationDynamics}{$D,r,\bfb,k$}
      \State \textbf{input:} Binary symmetric BOHT model $\BOHT(D,r,\bfb)$. Integer $k\ge 0$.
      \State \textbf{output:} BMS channel $P$ with $\#\supp(\mu_P) \le s+1$ satisfying $\BP^k(\Id) \le_{\ln} P$.
      \State $M_0 \gets \Id$
      \For{$i=1\to k$}
        \State $M_i \gets \txtsc{UpperBP}(D,r,\bfb,M_{i-1})$
      \EndFor
      \State \Return $M_k$.
    \EndProcedure

    \Procedure{UpperBP}{$D,r,\bfb,P$}
      \State \textbf{input:} Binary symmetric BOHT model $\BOHT(D,r,\bfb)$. Finite-support BMS channel $P$.
      \State \textbf{output:} BMS channel $Q$ with $\#\supp(\mu_Q) \le s+1$ satisfying $\BP(P) \le_{\ln} Q$.
      \State $P \gets \txtsc{UpperQuantize}(P)$
      \State \txtsc{BMSBuilder} \txtsf{builder}
      \Comment{\cref{alg:bms-builder}}
      \State $\txtsf{builder}.\txtsc{Initialize}()$
      \State $b^* \gets \min\{b\in \bZ_{\ge 0}: D\left( \bZ_{> b} \right) < w^{-1}\}$
      \State $P_0 \gets 0$, $P_1 \gets \txtsc{UpperQuantize}\left(P^{\times (r-1)} \circ B\right)$
      \For{$b=0\to b^*$}
        \If{$b\ge 2$}
          $P_b \gets \txtsc{UpperQuantize}\left(P_{\lfloor \frac b2\rfloor} \star P_{\lceil \frac b2 \rceil}\right)$
        \EndIf
        \State $\txtsf{builder}.\txtsc{AddBMS}(P_b, \lfloor D(b) w\rfloor w^{-1})$
      \EndFor
      \State \Return $\txtsf{builder}.\txtsc{Output}()$
    \EndProcedure

    \Procedure{UpperQuantize}{$P$}
      \State \textbf{input:} Finite-support BMS channel $P$
      \State \textbf{output:} BMS channel $Q$ with $\#\supp(\mu_Q) \le s+1$ satisfying $P \le_{\ln} Q$.
      \State \txtsc{BMSBuilder} \txtsf{builder}
      \Comment{\cref{alg:bms-builder}}
      \State $\txtsf{builder}.\txtsc{Initialize}()$
      \State $\txtsf{builder}.\txtsc{AddBMS}(P, 1)$
      \State \Return $\txtsf{builder}.\txtsc{Output}()$
    \EndProcedure
  \end{algorithmic}
\end{algorithm}

\begin{algorithm}[ht]
  \caption{BMS channel builder}
  \label{alg:bms-builder}
  \begin{algorithmic}[1]
    \State \textbf{global parameters:} Support size $s\in \bZ_{\ge 1}$. Precision parameter $w \in \bZ_{\ge 1}$.
    \State \textbf{data structure} \txtsc{BMSBuilder}
    \State \textbf{member:}
    $0 = \theta_0 < \cdots < \theta_s = 1$.
    $0\le p_0,\ldots,p_s\le 1$.
    \Procedure{Initialize}{$ $}
      \State $\theta_i \gets \frac is$, $\forall 0\le i\le s$
      \State $p_i \gets 0$, $\forall 0\le i\le s$
    \EndProcedure
    \Procedure{AddBSC}{$\theta,p$}
      \State \textbf{input:} $\theta,p\in [0, 1]$
      \State $i\gets \max\{0\le i\le s: \theta_i \le \theta\}$
      \If{$i=s$}
        \State $p_s \gets p_s + \lfloor p w \rfloor w^{-1}$
        \State \Return
      \EndIf
      \State $t \gets \frac{\theta^2-\theta_i^2}{\theta_{i+1}^2-\theta_i^2}$
      \State $p_i \gets p_i + \lfloor (1-t)p w \rfloor w^{-1}$
      \State $p_{i+1} \gets p_{i+1} + \lfloor t p w \rfloor w^{-1}$
  \EndProcedure
    \Procedure{AddBMS}{$P,c$}
      \State \textbf{input:} Finite-support BMS channel $P$. $c\in [0, 1]$.
      \For{$\Delta_P\in \supp(\mu_P)$}
        \State \txtsc{AddBSC}($1-2\Delta_P, c \cdot \mu_P(\Delta_P)$)
      \EndFor
    \EndProcedure
    \Procedure{Output}{$ $}
      \State $p_s \gets p_s + 1-\sum_{0\le i\le s} p_i$
      \Comment{Normalize}
      \State \Return $\sum_{0\le i\le s} p_i \BSC_{\frac{1-\theta_i}{2}}$
    \EndProcedure
  \end{algorithmic}
\end{algorithm}

\subsection{Overview} \label{sec:pop:overview}
We begin with a brief overview of the rigorous population dynamics algorithm.
The goal of the algorithm is to compute a BMS channel $P$ satisfying $\BP^k(\Id) \le_{\ln} P$.
There are two obvious choices of $P$, neither satisfactory here.
One obvious choice is $\BP^k(\Id)$. This would be the perfect choice if we could represent it and perform the relevant computations exactly on a computer. Unfortunately, real computers have finite computational resources, while $\BP^k(\Id)$ a priori could have countable support. Therefore, we need to find a finite-support upper bound of $\BP^k(\Id)$.
The other obvious choice is $P=\Id$, which is the largest channel in the less-noisy preorder, and has support size $1$. However, it is too far away from the actual channel $\BP^k(\Id)$.
We would like the output channel $P$ to be close enough to the actual channel $\BP^k(\Id)$ that it falls in the regime where robust reconstruction is impossible.

We now introduce the algorithm.
First suppose that infinite-precision arithmetic is available.
Let $s\ge 1$ be an integer and $0=\theta_0 < \theta_1 < \cdots < \theta_s=1$ be a sequence of real numbers.

For any finite-support BMS channel $P$, we can produce a BMS channel $Q$ such that $P\le_{\ln} Q$ and the $\theta$-component of $Q$ is contained in the set $\{\theta_i : 0\le i\le s\}$.
If $P=\BSC_\delta$, then we find an index $0\le i\le s$ such that $\theta = 1-2\delta \in [\theta_i, \theta_{i+1}]$, and output the channel
\begin{align} \label{eqn:pop:bsc-ub}
  Q = \frac{\theta_{i+1}^2 - \theta^2}{\theta_{i+1}^2 - \theta_i^2} \cdot \BSC_{\frac{1-\theta_i}2} + \frac{\theta^2 - \theta_i^2}{\theta_{i+1}^2 - \theta_i^2} \cdot \BSC_{\frac{1-\theta_{i+1}}2}.
\end{align}
Note that
\begin{align}
  \frac{\theta_{i+1}^2 - \theta^2}{\theta_{i+1}^2 - \theta_i^2} \cdot \theta_i^2 + \frac{\theta^2 - \theta_i^2}{\theta_{i+1}^2 - \theta_i^2} \cdot \theta_{i+1}^2 = \theta^2.
\end{align}
So by \cref{lem:bms-ln-coupling}, this choice of $Q$ satisfies $P\le_{\ln} Q$.
If $P$ is a finite mixture of BSC channels $\{\BSC_{\delta_j} : 1\le j\le s'\}$, then we can apply the above procedure for each $\BSC_{\delta_j}$ and take the mixture of the outputs. In this way, the final output channel has $\theta$-component contained in $\{\theta_i : 0\le i\le s\}$, and is less noisy than $P$ because $\le_{\ln}$ respects mixtures.

The same idea handles BMS channels $P$ with countable support.
Suppose $P = \sum_{j\ge 1} p_j \BSC_{\delta_j}$. Then we can truncate $P$ using
\begin{align}
  P \le_{\deg} \left(\sum_{1\le j\le s'} p_j \BSC_{\delta_j}\right) + \left(1- \sum_{1\le j\le s'} p_j\right) \Id = P',
\end{align}
where $s'$ is some integer,
and then apply the above procedure to produce an upper bound channel for $P'$.

Recall that $\BP(P) = \bE_{b\sim D} \left( P^{\times (r-1)} \circ B \right)^{\star b}$.
Suppose that $P$ has finite support. Then we can compute a small-support BMS channel upper bound for $P^{\times (r-1)} \circ B$.
Let $P_0 = 0$ and $P_1$ be this upper bound.
Take a large $b^*$.
For every $2\le b\le b^*$, we let $P_b$ be a small-support BMS channel upper bound for $P_{\lfloor \frac b2\rfloor} \star P_{\lceil \frac b2 \rceil}$.
Then $P_b$ is also a small-support upper bound for $\left(P^{\times (r-1)} \circ B\right)^{\star b}$.
Finally,
\begin{align}
  Q = \left(\sum_{0\le b\le b^*} D(b) P_b\right) + D(\bZ_{>b^*}) \Id
\end{align}
is a small-support upper bound for $\BP(P)$.
Iterating the procedure $k$ times, we get a small-support upper bound for $\BP^k(\Id)$.

We now remove the infinite-precision arithmetic assumption.
This requires care: even if the input numbers have low complexity (i.e., can be represented using a small number of bits), intermediate values may require many bits to represent.
For example, $\frac{\theta_{i+1}^2-\theta^2}{\theta_{i+1}^2-\theta_i^2}$, which appeared in \cref{eqn:pop:bsc-ub}, a priori could have higher complexity than $\theta_i$, $\theta_{i+1}$, and $\theta$.
We cannot simply round the numbers because that may break the less-noisy preorder. Nevertheless, with careful handling, all numbers appearing in the computation have bounded complexity, and the less-noisy preorder is preserved.

The final algorithm is presented in \cref{alg:pop,alg:bms-builder}.

\subsection{Analysis} \label{sec:pop:analysis}
In this section we analyze the rigorous population dynamics algorithm. In particular, we prove that the output of \txtsc{PopulationDynamics} is always a small-support upper bound of $\BP^k(\Id)$ with small-complexity numbers.
We do not give guarantees for how close the upper bound is to the actual channel, but in \cref{sec:special-boht} we show that it is close enough for our purposes.

We use the following notion to measure the complexity of a number.
\begin{definition}[Numerical safety] \label{defn:pop-safety}
  Let $w\in \bZ_{\ge 1}$.
  A real number $x\in [0, 1]$ is called $w$-safe if $x u \in \bZ$ for some $u\in [w]$.

  Let $s,w\in \bZ_{\ge 1}$.
  A finite-support BMS channel $P$ is called $(s,w)$-safe if $\#\supp(\mu_P) \le s+1$, and for all $\Delta_P \in \supp(\mu_P)$, $1-2\Delta_P$ and $\mu_P(\Delta_P)$ are both $w$-safe.
\end{definition}
Note that arithmetic operations involving $w$-safe numbers can be performed in $\poly(\log w)$ time.

Our main result in this section is the following.
\begin{proposition}[\txtsc{PopulationDynamics}] \label{prop:pop}
  Let $s\le w\in \bZ_{\ge 1}$ be parameters.
  Let $s',w'\in \bZ_{\ge 1}$ be such that $s'\le w'$.
  Suppose that $\txtsc{PopulationDynamics}(D,r,\bfb,k)$ is called where
  \begin{enumerate}[label=(\arabic*)]
    \item $b_i$ is $w'$-safe for all $0\le i\le r-1$;
    \item for any $b\in \bZ_{\ge 0}$, $\lfloor D(b) w\rfloor$ can be computed using $O(1)$ arithmetic operations involving $w$-safe numbers.
  \end{enumerate}
  Let $b^* = \min\{b\in \bZ_{\ge 0}: D\left( \bZ_{> b} \right) < w^{-1}\}$.
  Let $Q$ be the channel returned.
  Then the following hold.
  \begin{enumerate}[label=(\roman*)]
    \item\label{item:prop:pop:i} $Q$ is a $(s,w)$-safe BMS channel.
    \item\label{item:prop:pop:ii} $\BP^k(\Id) \le_{\ln} Q$.
    \item\label{item:prop:pop:iii} The total computation cost is $O\left(k s^{r-1}+k s^2 b^*\right)$ arithmetic operations involving $(ww')^{O(1)}$-safe numbers.
  \end{enumerate}
\end{proposition}

The proof of \cref{prop:pop} is by analyzing its components. The core of the algorithm is the \txtsc{BMSBuilder} data structure (\cref{alg:bms-builder}).
\begin{lemma}[\txtsc{BMSBuilder}] \label{lem:pop-bms-builder}
  Let \txtsf{builder} be a \txtsc{BMSBuilder} instance initialized with parameters $s\le w\in \bZ_{\ge 1}$.
  Let $s',w'\in \bZ_{\ge 1}$ be such that $s'\le w'$.
  Suppose that a sequence of $m$ $\txtsf{builder}.\txtsc{AddBMS}$ calls are made, with arguments $(P_i, c_i)$ ($i\in [m]$), where
  \begin{enumerate}[label=(\arabic*)]
    \item $P_i$ is a $(s',w')$-safe BMS channel;
    \item $c_i$ is $w'$-safe, $0\le c_i\le 1$, and $\sum_{i\in [m]} c_i \le 1$.
  \end{enumerate}
  Let $Q$ be the channel returned by $\txtsf{builder}.\txtsc{Output}()$.
  Then the following hold.
  \begin{enumerate}[label=(\roman*)]
    \item \label{item:lem:pop-bms-builder:i} $Q$ is a $(s,w)$-safe BMS channel.
    \item \label{item:lem:pop-bms-builder:ii} For any BMS channel $R$ and
    \begin{align}
      P = \left(\sum_{i\in [m]} c_i P_i\right) + \left(1-\sum_{i\in [m]} c_i \right) R,
    \end{align}
    we have $P \le_{\ln} Q$.
    \item \label{item:lem:pop-bms-builder:iii} The total computation cost is $O(m s')$ arithmetic operations involving $(ww')^{O(1)}$-safe numbers.
  \end{enumerate}
\end{lemma}
\begin{proof}
  Suppose $P_i = \sum_{1\le j\le s_i} p_{i,j} \BSC_{\delta_{i,j}}$ ($i\in [m]$).
  Then \txtsc{AddBSC} is called with $(\theta = 1-2\delta_{i,j}, p = c_i p_{i,j})$ for $i\in [m]$, $j\in [s_i]$.
  During the call corresponding to $(i,j)$, suppose we performed update $p_\ell \gets p_\ell + q_{i,j,\ell}$ for $\ell\in T_{i,j}$, for some subset $T_{i,j} \subseteq \{0,\ldots,s\}$. Note that $T_{i,j}$ is either $\{s\}$ or $\{\ell,\ell+1\}$ for some $0\le \ell\le s-1$.
  Let $r_{i,j} = c_i p_{i,j} - \sum_{\ell\in T_{i,j}} q_{i,j,\ell}$. (Note that $r_{i,j}$ is not necessarily $w$-safe and is only for analysis purposes.)
  For any number $x\in [0, 1]$, we have $\lfloor x w \rfloor w^{-1} \le x$.
  So
  \begin{align}
    \sum_{\ell\in T_{i,j}} q_{i,j,\ell} &\le c_i p_{i,j},
  \end{align}
  and thus $r_{i,j} \ge 0$.
  Furthermore,
  \begin{align} \label{eqn:lem:pop-bms-builder:coupling}
    \sum_{\ell\in T_{i,j}} q_{i,j,\ell} \theta_\ell^2 + r_{i,j} \ge c_i p_{i,j} (1-2\delta_{i,j})^2.
  \end{align}

  Let $q_\ell$ ($0\le \ell \le s$) denote the value of $p_\ell$ just before the \txtsc{Output} normalization.
  Then $q_\ell = \sum_{i\in [m], j\in [s_i]} q_{i,j,\ell}$. (When $\ell\not \in T_{i,j}$ we let $q_{i,j,\ell}=0$.)
  The normalization adds
  \begin{align}
    1-\sum_{0\le \ell\le s} q_\ell
    = 1-\sum_{i\in [m]} c_i + \sum_{i\in [m],j\in [s_i]} r_{i,j}
  \end{align}
  to $p_s$.
  Let $Q$ denote the final output channel.
  Because $q_{i,j,\ell} w \in \{0,\ldots,w\}$ for all $i\in [m], j\in [s_i], \ell \in T_{i,j}$, and because the normalization amount is also a multiple of $w^{-1}$, the final weights of $Q$ are all $w$-safe.
  Therefore the output channel $Q$ is $(s,w)$-safe. This proves \cref{item:lem:pop-bms-builder:i}.

  Consider the coupling where for every $i\in [m]$, $j\in [s_i]$, we couple $(1-2\delta_{i,j}, \theta_\ell)$ with probability $q_{i,j,\ell}$ for $\ell\in T_{i,j}$, and couple $(1-2\delta_{i,j}, 1)$ with probability $r_{i,j}$. Finally, we couple $(1,1)$ with probability $1-\sum_{i\in [m]} c_i$.
  By \cref{eqn:lem:pop-bms-builder:coupling} and \cref{lem:bms-ln-coupling}, we have
  \begin{align}
    P &= \left(\sum_{i\in [m]} c_i P_i\right) + \left( 1-\sum_{i\in [m]} c_i \right) R \\
    \nonumber &\le_{\deg} \left(\sum_{i\in [m]} c_i P_i\right) + \left( 1-\sum_{i\in [m]} c_i \right) \Id \\
    \nonumber & = \left(\sum_{i\in [m]} \sum_{j\in [s_i]} c_i p_{i,j} \BSC_{\delta_{i,j}} \right) + \left( 1-\sum_{i\in [m]} c_i \right) \Id \\
    \nonumber & \le_{\ln} \left(\sum_{i\in [m]} \sum_{j\in [s_i]} \left( \sum_{\ell\in T_{i,j}} q_{i,j,\ell} \BSC_{\frac{1-\theta_\ell}2} + r_{i,j} \Id \right) \right) + \left( 1-\sum_{i\in [m]} c_i \right) \Id \\
    \nonumber & = \sum_{0\le \ell<s} q_\ell \BSC_{\frac{1-\theta_\ell}2}
    + \left(q_s + \sum_{i\in [m],j\in [s_i]} r_{i,j} + 1-\sum_{i\in [m]} c_i\right) \Id
    = Q.
  \end{align}
  This proves \cref{item:lem:pop-bms-builder:ii}.

  During the execution, all intermediate values are $(w w')^{O(1)}$-safe.
  This proves \cref{item:lem:pop-bms-builder:iii}.
\end{proof}

\begin{lemma}[\txtsc{UpperQuantize}] \label{lem:pop-restrict}
  Let $s\le w\in \bZ_{\ge 1}$ be parameters.
  Let $s',w'\in \bZ_{\ge 1}$ be such that $s'\le w'$.
  Suppose that $\txtsc{UpperQuantize}(P)$ is called where $P$ is a $(s',w')$-safe BMS channel.
  Let $Q$ be the channel returned.
  Then the following hold.
  \begin{enumerate}[label=(\roman*)]
    \item $Q$ is a $(s,w)$-safe BMS channel.
    \item $P \le_{\ln} Q$.
    \item The total computation cost is $O(s')$ arithmetic operations involving $(ww')^{O(1)}$-safe numbers.
  \end{enumerate}
\end{lemma}
\begin{proof}
  This follows from \cref{lem:pop-bms-builder} with $m=1$.
\end{proof}

\begin{lemma}[\txtsc{UpperBP}] \label{lem:pop-bp}
  Let $s\le w\in \bZ_{\ge 1}$ be parameters.
  Let $s',w'\in \bZ_{\ge 1}$ be such that $s'\le w'$.
  Suppose that $\txtsc{UpperBP}(D,r,\bfb,P)$ is called where
  \begin{enumerate}[label=(\arabic*)]
    \item $P$ is a $(s',w')$-safe BMS channel;
    \item $b_i$ is $w'$-safe for all $0\le i\le r-1$;
    \item for any $b\in \bZ_{\ge 0}$, $\lfloor D(b) w\rfloor$ can be computed using $O(1)$ arithmetic operations involving $w$-safe numbers.
  \end{enumerate}
  Let $b^* = \min\{b\in \bZ_{\ge 0}: D\left( \bZ_{> b} \right) < w^{-1}\}$.
  Let $Q$ be the channel returned.
  Then the following hold.
  \begin{enumerate}[label=(\roman*)]
    \item\label{item:lem:pop-bp:i} $Q$ is a $(s,w)$-safe BMS channel.
    \item\label{item:lem:pop-bp:ii} $\BP(P) \le_{\ln} Q$.
    \item\label{item:lem:pop-bp:iii} The total computation cost is $O\left(s^{r-1}+s^2 b^*\right)$ arithmetic operations involving $(ww')^{O(1)}$-safe numbers.
  \end{enumerate}
\end{lemma}
\begin{proof}
  \cref{item:lem:pop-bp:i} follows from \cref{lem:pop-bms-builder}.

  By induction on $0\le b\le b^*$ and using \cref{lem:pop-restrict}, we have
  \begin{align}
    P_b \ge_{\ln} \left( P^{\times (r-1)} \circ B \right)^{\star b}, \qquad 0\le b\le b^*.
  \end{align}
  Let $c_b = \lfloor D(b) w \rfloor w^{-1}$.
  By \cref{lem:pop-bms-builder} we have
  \begin{align}
    \BP(P) &= \sum_{b\ge 0} D(b) \left( P^{\times (r-1)} \circ B \right)^{\star b} \\
    \nonumber &\le_{\deg} \left(\sum_{0\le b\le b^*} D(b) \left( P^{\times (r-1)} \circ B \right)^{\star b} \right) + D(\bZ_{>b^*}) \Id \\
    \nonumber &\le_{\ln} \left(\sum_{0\le b\le b^*} D(b) P_b\right) + D(\bZ_{>b^*}) \Id \\
    \nonumber &\le_{\deg} \left(\sum_{0\le b\le b^*} c_b P_b\right) + \left( 1-\sum_{0\le b\le b^*} c_b\right) \Id \\
    \nonumber &\le_{\ln} Q,
  \end{align}
  where $Q$ is the final output channel.
  This proves \cref{item:lem:pop-bp:ii}.

  Computation of $\txtsc{UpperQuantize}(P^{\times (r-1)} \circ B)$ takes $O(s^{r-1})$ arithmetic operations involving $(ww')^{O(1)}$-safe numbers.
  Computation of each $P_b$ ($2\le b\le b^*$) takes $O(s^2)$ arithmetic operations involving $(ww')^{O(1)}$-safe numbers.
  This proves \cref{item:lem:pop-bp:iii}.
\end{proof}

\begin{proof}[Proof of \cref{prop:pop}]
  \cref{item:prop:pop:i,item:prop:pop:iii} follow from the corresponding claims in \cref{lem:pop-bp}.

  By induction on $0\le i\le k$ and using \cref{lem:pop-bp}, we have
  \begin{align}
    M_i \ge_{\ln} \BP^i(\Id), \qquad 0\le i\le k.
  \end{align}
  This proves \cref{item:prop:pop:ii}.
\end{proof}

\section{Reconstruction threshold for the special BOHT} \label{sec:special-boht}
In this section we apply the methods from \cref{sec:robust,sec:pop} to $\BOHT(\Pois(d),r,\lambda)$, the special BOHT model on a Galton-Watson hypertree with Poisson offspring distribution, and prove \cref{thm:boht}.

\subsection{Robust non-reconstruction} \label{sec:special-boht:robust}
Recall that for the special BOHT model,
\begin{align}
  b_k = \left\{
    \begin{array}{ll}
      \frac{1-\lambda}{2^{r-1}}, & \text{if}~0\le k\le r-2,\\
      \lambda + \frac{1-\lambda}{2^{r-1}}, & \text{if}~k=r-1.
    \end{array}
  \right.
\end{align}
We compute the function $f$ defined in \cref{eqn:prop:robust:f}.
For $0\le j\le i\le r-1$, we have
\begin{align}
  p_{i,j} = \sum_{0\le k\le r-1} b_k \binom ij \binom{r-1-i}{k-j}
  = \binom ij 2^{-i} (1-\lambda) + \lambda \mathbbm{1}\{i=j\}.
\end{align}
So for $1\le i\le r-1$,
\begin{align}
  c_i = \frac{(p_{i,i}-p_{i,0})^2}{p_{i,i}+p_{i,0}} = \frac{\lambda^2}{\lambda + 2^{1-i} (1-\lambda)}.
\end{align}
Then
\begin{align}
  g(x) = \sum_{1\le i\le r-1} \binom{r-1}i x^i (1-x)^{r-1-i} \cdot \frac{\lambda^2}{\lambda + 2^{1-i} (1-\lambda)}.
\end{align}
Finally, for $D = \Pois(d)$,
\begin{align} \label{eqn:special-boht:f}
  f(x) &= 1-\exp(-d g(x)) \\
  \nonumber &= 1-\exp\left( -d \sum_{1\le i\le r-1} \binom{r-1}i x^i (1-x)^{r-1-i} \cdot \frac{\lambda^2}{\lambda + 2^{1-i} (1-\lambda)} \right).
\end{align}

The following proposition was previously proved in \cite{gu2024community} using a Taylor expansion analysis of the BP operator. The proof below is considerably simpler and gives a better robustness parameter $\epsilon$.
\begin{proposition} \label{prop:special-robust}
  Fix $r\ge 3$. There exists $\epsilon=\epsilon(r)>0$ such that the following holds.
  For any $\lambda\in \left[-\frac1{2^{r-1}-1}, 1\right]$ and $d$ such that $(r-1) d \lambda^2 \le 1$, for the model $\BOHT(\Pois(d),r,\lambda)$, robust reconstruction is impossible with respect to all BMS channels $W$ with $C_{\chi^2}(W) \le \epsilon$.
\end{proposition}
\begin{proof}
  For $(r-1) d \lambda^2 = 1$, \cref{eqn:special-boht:f} simplifies to
  \begin{align} \label{eqn:special-boht:f-r-lm}
    f(x) = f_{r,\lambda}(x) &= 1-\exp\left( - \frac 1{r-1} \sum_{1\le i\le r-1} \binom{r-1}i x^i (1-x)^{r-1-i} \cdot \frac{1}{\lambda + 2^{1-i} (1-\lambda)} \right).
  \end{align}
  Note that for fixed $r,x$, $f_{r,\lambda}(x)$ is non-increasing in $\lambda\in \left[ -\frac 1{2^{r-1}-1}, 1 \right]$.
  In particular, $f_{r,\lambda}(x) \le f_{r,-\frac 1{2^{r-1}-1}}(x)$.

  For $r=3$, we have
  \begin{align} \label{eqn:special-boht:f-r-lm-3}
    f_{r,-\frac 1{2^{r-1}-1}}(x) = f_{3,-\frac 13}(x) = 1-\exp\left(-x-\frac 12 x^2\right) =
    x - \frac 13 x^3 + O_r(x^4).
  \end{align}
  For $r\ge 4$, we have
  \begin{align}
    f_{r,-\frac 1{2^{r-1}-1}}(x) &= 1-\exp\left( - x(1-(r-2)x) - \frac{r-2}2 \cdot x^2 \cdot \frac {2^{r-1}-1}{2^{r-2}-1} + O_r(x^3) \right) \\
    \nonumber &= x - \frac{2^{r-2}-r+1}{2^{r-1}-2}\cdot x^2+ O_r(x^3).
  \end{align}
  Therefore, for any $r\ge 3$, there exists $\epsilon=\epsilon(r) > 0$ such that $f_{r,\lambda}(x) < x$ for all $0 < x \le \epsilon$.
  By \cref{prop:robust}, for all BMS channels $P$ with $0<C_{\chi^2}(P) \le \epsilon$, we have
  \begin{align}
    C_{\chi^2}(\BP(P)) \le f\left(C_{\chi^2}(P)\right) < C_{\chi^2}(P).
  \end{align}
  Therefore for BMS channels $W$ with $C_{\chi^2}(W) \le \epsilon$, we have
  \begin{align}
    \lim_{k\to \infty} C_{\chi^2}\left(\BP^k(W)\right) = 0,
  \end{align}
  and robust reconstruction is impossible with respect to $W$.
\end{proof}

\subsection{Case \texorpdfstring{$r=3$}{r=3}} \label{sec:special-boht:r3}
\begin{figure}[h!]
  \centering
  \includegraphics[width=0.5\textwidth]{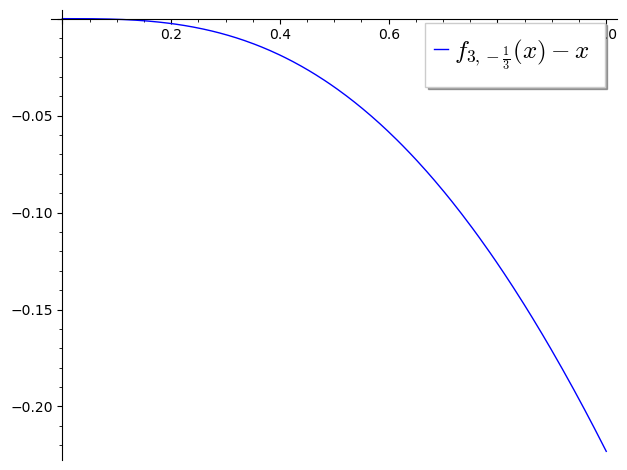}
  \caption{Plot of $f_{3,-\frac 13}(x)-x$ for $x\in [0, 1]$.}
  \label{fig:boht-r3}
\end{figure}

For $r=3$, any $\lambda\in \left[-\frac 13, 1\right]$, and $0<x\le 1$, we have
\begin{align}
  f_{r,\lambda}(x) \le f_{3,-\frac 13}(x)
  = 1-\exp\left(-x-\frac 12 x^2\right) < x,
\end{align}
where the last step follows from
\begin{align}
  \log(1-x) = -\sum_{i\ge 1} \frac{x^i}{i} < -x - \frac 12 x^2, \qquad 0<x\le 1.
\end{align}
\cref{fig:boht-r3} shows the plot of $f_{3,-\frac 13}(x)-x$ for $x\in [0, 1]$.

For any non-trivial BMS channel $P$, we have
\begin{align}
  C_{\chi^2}(\BP(P)) \le f\left(C_{\chi^2}(P)\right) < C_{\chi^2}(P).
\end{align}
Therefore
\begin{align}
  \lim_{k\to \infty} C_{\chi^2}\left(\BP^k(\Id)\right) = 0
\end{align}
and reconstruction is impossible below the KS threshold for $r=3$ and any $\lambda\in \left[-\frac 13, 1\right]$.
This proves the $r=3$ case of \cref{thm:boht}.

\subsection{Case \texorpdfstring{$r=4$}{r=4}} \label{sec:special-boht:r4}
\begin{figure}[h!]
  \centering
  \begin{subfigure}{0.5\textwidth}
    \centering
    \includegraphics[width=\linewidth]{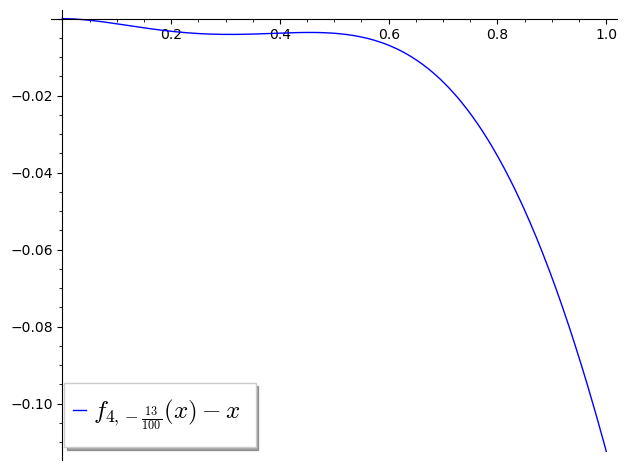}
    \caption{Plot of $f_{4,-\frac{13}{100}}(x)-x$ for $x\in [0, 1]$.}
    \label{fig:boht-r4:a}
  \end{subfigure}%
  \begin{subfigure}{0.5\textwidth}
    \centering
    \includegraphics[width=\linewidth]{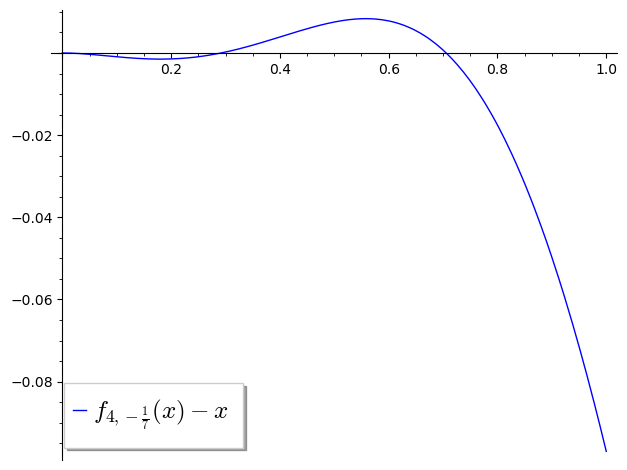}
    \caption{Plot of $f_{4,-\frac 17}(x)-x$ for $x\in [0, 1]$.}
    \label{fig:boht-r4:b}
  \end{subfigure}
  \begin{subfigure}{0.5\textwidth}
    \centering
    \includegraphics[width=\linewidth]{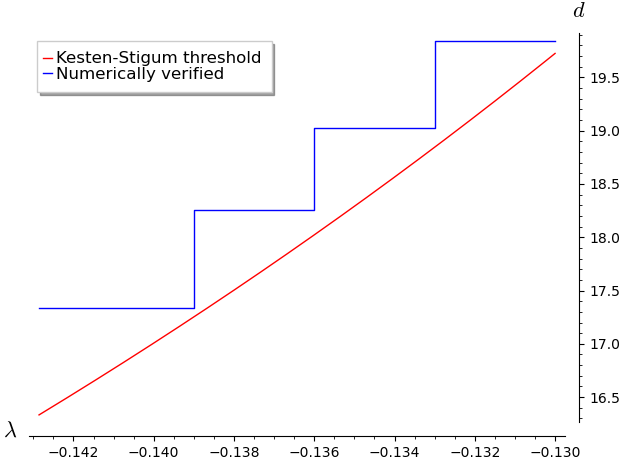}
    \caption{Plot of execution results for $r=4$. Red line is the Kesten-Stigum threshold $(r-1) d \lambda^2 = 1$. Parameters below the blue line are certified by the rigorous computation to satisfy $\lim_{k\to \infty} C_{\chi^2}\left(\BP^k(\Id)\right) \le 0.2$. See also \cref{tab:boht-r4}.}
    \label{fig:boht-r4:c}
  \end{subfigure}
  \caption{Plots for $r=4$.}
  \label{fig:boht-r4}
\end{figure}

\begin{table}[h!]
  \centering
  \begin{tabular}{|c|c|c|c|c|}
    \hline $\lambda$ & $-\frac{1}{7}$ & $-\frac{139}{1000}$ & $-\frac{136}{1000}$ & $-\frac{133}{1000}$ \\
    \hline $d-\frac 1{(r-1)\lambda^2}$ & $1$ & $1$ & $1$ & $1$ \\
    \hline $k$ & $47$ & $33$ & $27$ & $24$ \\
    \hline
  \end{tabular}
  \caption{Execution results for $r=4$. Support size $s=8$. Precision parameter $w=2^{32}$. See also \cref{fig:boht-r4:c,app:code}.}
  \label{tab:boht-r4}
\end{table}

We take $\lambda_0 = -\frac{13}{100}$ and divide into two cases $\lambda\in [\lambda_0, 1]$ and $\lambda\in \left[ -\frac 17, \lambda_0\right]$.

\paragraph{Case $\lambda\in [\lambda_0, 1]$}
For any $\lambda\in [\lambda_0, 1]$ and $0<x\le 1$, we have
\begin{align}
  f_{r,\lambda}(x) \le f_{r,\lambda_0}(x) < x,
\end{align}
where the last step follows from
\begin{align} \label{eqn:boht-r4:num}
  -\log\left(1-f_{r,\lambda_0}(x)\right) = x + \frac{26}{87} x^{2} + \frac{1569}{1769} x^{3}
  < \sum_{1\le i\le 10} \frac{x^i}{i} < -\log(1-x), \qquad \forall 0<x\le 1.
\end{align}
The middle inequality in \cref{eqn:boht-r4:num} is an inequality between polynomials and can be proved rigorously. In the following, whenever we use an inequality of the form $f_{r,\lambda}(x) < x$ for a given $\lambda$, we prove it by the same method.
See \cref{fig:boht-r4:a} for a plot of $f_{r,\lambda_0}(x)-x$.
Therefore reconstruction is impossible below the KS threshold for $\lambda\in [\lambda_0, 1]$.

\paragraph{Case $\lambda\in \left[ -\frac 17, \lambda_0\right]$}
For $0<x\le 0.2$, we have
\begin{align}
  f_{r,\lambda}(x) \le f_{4,-\frac 17}(x) < x.
\end{align}
See \cref{fig:boht-r4:b} for a plot of $f_{4,-\frac 17}(x)-x$.

Thus robust reconstruction is impossible for all BMS channels $W$ satisfying $C_{\chi^2}(W) \le 0.2$. It remains to prove that for any $\lambda\in \left[ -\frac 17, \lambda_0\right]$ and $(r-1) d \lambda^2 = 1$, there exists $k\in \bZ_{\ge 0}$ such that
\begin{align} \label{eqn:special-boht:r4:num}
  C_{\chi^2}\left(\BP^k(\Id)\right) \le 0.2.
\end{align}
We take a set $T$ of tuples $(\lambda,d,k)$ as shown in \cref{tab:boht-r4} and run \cref{alg:pop} for each of these tuples.
The computations show that \cref{eqn:special-boht:r4:num} holds for all $(\lambda,d,k)\in T$.
Furthermore, for any $\lambda\in \left[ -\frac 17, \lambda_0\right]$ and $(r-1) d \lambda^2 = 1$, there exists $(\lambda',d',k')\in T$ such that $\lambda'\le \lambda$ and $d' \ge d$.
By \cref{lem:boht-monotone},
\begin{align}
  C_{\chi^2}\left( \BP_{r,d,\lambda}^{k'}(\Id) \right)
  \le C_{\chi^2}\left( \BP_{r,d',\lambda'}^{k'}(\Id) \right)
  \le 0.2.
\end{align}
See \cref{fig:boht-r4:c} for a comparison of the set $T$ and the Kesten-Stigum threshold.
This proves \cref{eqn:special-boht:r4:num}.
Combining \cref{eqn:special-boht:r4:num} with the robust non-reconstruction result, reconstruction is impossible below the KS threshold for $\lambda\in \left[ -\frac 17, \lambda_0\right]$.

\paragraph{Summary}
Combining the two cases, reconstruction is impossible below the KS threshold for $r=4$ and any $\lambda\in \left[-\frac 17, 1\right]$.
This proves the $r=4$ case of \cref{thm:boht}.

\subsection{Case \texorpdfstring{$r=5$}{r=5}} \label{sec:special-boht:r5}
\begin{figure}[h!]
  \centering
  \begin{subfigure}{0.5\textwidth}
    \centering
    \includegraphics[width=\linewidth]{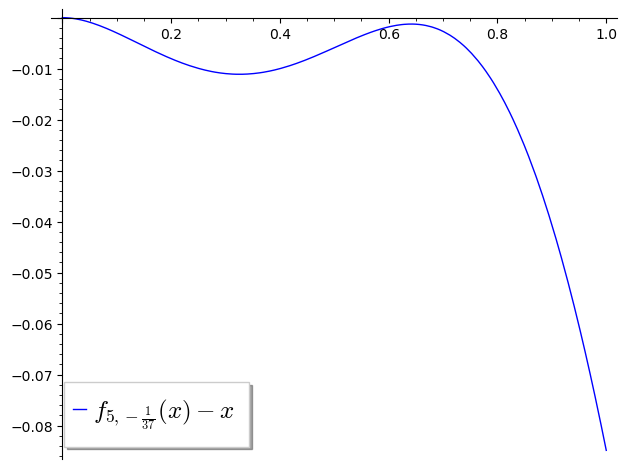}
    \caption{Plot of $f_{5,-\frac 1{37}}(x)-x$ for $x\in [0, 1]$.}
    \label{fig:boht-r5:a}
  \end{subfigure}%
  \begin{subfigure}{0.5\textwidth}
    \centering
    \includegraphics[width=\linewidth]{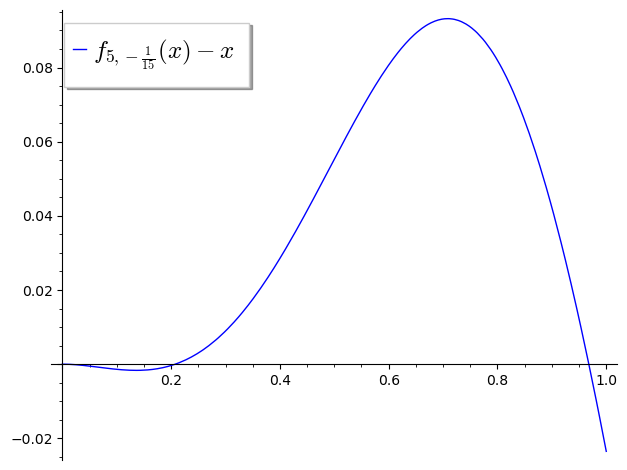}
    \caption{Plot of $f_{5,-\frac 1{15}}(x)-x$ for $x\in [0, 1]$.}
    \label{fig:boht-r5:b}
  \end{subfigure}
  \begin{subfigure}{0.5\textwidth}
    \centering
    \includegraphics[width=\linewidth]{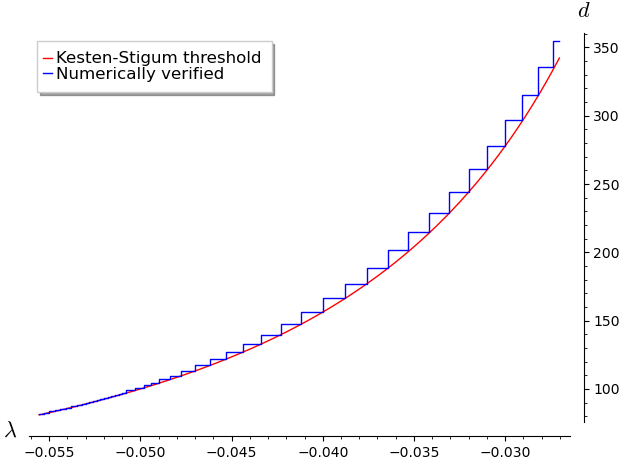}
    \caption{Plot of execution results for $r=5$. Red line is the Kesten-Stigum threshold $(r-1) d \lambda^2 = 1$. Parameters below the blue line are certified by the rigorous computation to satisfy $\lim_{k\to \infty} C_{\chi^2}\left(\BP^k(\Id)\right) \le 0.2$. See also \cref{tab:boht-r5}.}
    \label{fig:boht-r5:c}
  \end{subfigure}
  \caption{Plots for $r=5$.}
  \label{fig:boht-r5}
\end{figure}

\begin{table}[h!]
  \centering
  \begin{tabular}{|c|c|c|c|c|c|c|c|c|}
    \hline $\lambda$ & $-\frac{1}{18}$ & $-\frac{553}{10000}$ & $-\frac{550}{10000}$ & $-\frac{547}{10000}$ & $-\frac{544}{10000}$ & $-\frac{541}{10000}$ & $-\frac{538}{10000}$ & $-\frac{535}{10000}$ \\
    \hline $d-\frac 1{(r-1)\lambda^2}$ & $1$ & $1$ & $1$ & $1$ & $1$ & $1$ & $1$ & $1$ \\
    \hline $k$ & $68$ & $61$ & $54$ & $50$ & $46$ & $43$ & $41$ & $39$ \\
    \hline \hline $\lambda$ & $-\frac{532}{10000}$ & $-\frac{530}{10000}$ & $-\frac{528}{10000}$ & $-\frac{526}{10000}$ & $-\frac{524}{10000}$ & $-\frac{522}{10000}$ & $-\frac{520}{10000}$ & $-\frac{518}{10000}$ \\
    \hline $d-\frac 1{(r-1)\lambda^2}$ & $1$ & $1$ & $1$ & $1$ & $1$ & $1$ & $1$ & $1$ \\
    \hline $k$ & $37$ & $36$ & $35$ & $35$ & $34$ & $33$ & $32$ & $32$ \\
    \hline \hline $\lambda$ & $-\frac{516}{10000}$ & $-\frac{514}{10000}$ & $-\frac{512}{10000}$ & $-\frac{510}{10000}$ & $-\frac{508}{10000}$ & $-\frac{503}{10000}$ & $-\frac{498}{10000}$ & $-\frac{494}{10000}$ \\
    \hline $d-\frac 1{(r-1)\lambda^2}$ & $1$ & $1$ & $1$ & $1$ & $2$ & $2$ & $2$ & $2$ \\
    \hline $k$ & $31$ & $30$ & $30$ & $29$ & $32$ & $31$ & $30$ & $29$ \\
    \hline \hline $\lambda$ & $-\frac{490}{10000}$ & $-\frac{484}{10000}$ & $-\frac{478}{10000}$ & $-\frac{470}{10000}$ & $-\frac{462}{10000}$ & $-\frac{453}{10000}$ & $-\frac{444}{10000}$ & $-\frac{434}{10000}$ \\
    \hline $d-\frac 1{(r-1)\lambda^2}$ & $3$ & $3$ & $4$ & $4$ & $5$ & $5$ & $6$ & $7$ \\
    \hline $k$ & $31$ & $29$ & $31$ & $29$ & $30$ & $28$ & $29$ & $29$ \\
    \hline \hline $\lambda$ & $-\frac{423}{10000}$ & $-\frac{412}{10000}$ & $-\frac{400}{10000}$ & $-\frac{388}{10000}$ & $-\frac{376}{10000}$ & $-\frac{364}{10000}$ & $-\frac{353}{10000}$ & $-\frac{342}{10000}$ \\
    \hline $d-\frac 1{(r-1)\lambda^2}$ & $8$ & $9$ & $10$ & $11$ & $12$ & $13$ & $14$ & $15$ \\
    \hline $k$ & $29$ & $29$ & $28$ & $28$ & $27$ & $26$ & $26$ & $25$ \\
    \hline \hline $\lambda$ & $-\frac{331}{10000}$ & $-\frac{320}{10000}$ & $-\frac{310}{10000}$ & $-\frac{300}{10000}$ & $-\frac{291}{10000}$ & $-\frac{282}{10000}$ & $-\frac{274}{10000}$ &  \\
    \hline $d-\frac 1{(r-1)\lambda^2}$ & $16$ & $17$ & $18$ & $19$ & $20$ & $21$ & $22$ &  \\
    \hline $k$ & $24$ & $24$ & $23$ & $22$ & $22$ & $21$ & $21$ &  \\
    \hline
  \end{tabular}
  \caption{Execution results for $r=5$. Support size $s=8$. Precision parameter $w=2^{32}$. See also \cref{fig:boht-r5:c,app:code}.}
  \label{tab:boht-r5}
\end{table}

We take $\lambda_0 = -\frac 1{37}$, $\lambda^*=-\frac{1}{18}$ and divide into two cases $\lambda\in [\lambda_0, 1]$ and $\lambda\in [\lambda^*, \lambda_0]$.

\paragraph{Case $\lambda\in [\lambda_0, 1]$}
For any $\lambda\in [\lambda_0, 1]$ and $0<x\le 1$ we have
\begin{align}
  f_{r,\lambda}(x) \le f_{r,\lambda_0}(x) < x.
\end{align}
See \cref{fig:boht-r5:a} for a plot of $f_{r,\lambda_0}(x)-x$.
Therefore reconstruction is impossible below the KS threshold for $\lambda\in [\lambda_0, 1]$.

\paragraph{Case $\lambda\in [\lambda^*, \lambda_0]$}
For $0<x\le 0.2$, we have
\begin{align}
  f_{r,\lambda}(x) \le f_{5,-\frac1{15}}(x) < x.
\end{align}
See \cref{fig:boht-r5:b} for a plot of $f_{5,-\frac 1{15}}(x)-x$.

Thus robust reconstruction is impossible for all BMS channels $W$ satisfying $C_{\chi^2}(W)\le 0.2$.
It remains to prove that for any $\lambda\in [\lambda^*, \lambda_0]$ and $(r-1) d \lambda^2 = 1$, there exists $k\in \bZ_{\ge 0}$ such that
\begin{align} \label{eqn:special-boht:r5:num}
  C_{\chi^2}\left(\BP^k(\Id)\right) \le 0.2.
\end{align}
We take a set $T$ of tuples $(\lambda,d,k)$ as shown in \cref{tab:boht-r5} and run \cref{alg:pop} for each of these tuples.
The computations show that \cref{eqn:special-boht:r5:num} holds for all $(\lambda,d,k)\in T$.
Furthermore, for any $\lambda\in [\lambda^*, \lambda_0]$ and $(r-1) d \lambda^2 = 1$, there exists $(\lambda',d',k')\in T$ such that $\lambda'\le \lambda$ and $d' \ge d$.
By \cref{lem:boht-monotone},
\begin{align}
  C_{\chi^2}\left( \BP_{r,d,\lambda}^{k'}(\Id) \right)
  \le C_{\chi^2}\left( \BP_{r,d',\lambda'}^{k'}(\Id) \right)
  \le 0.2.
\end{align}
See \cref{fig:boht-r5:c} for a comparison of the set $T$ and the Kesten-Stigum threshold.
This proves \cref{eqn:special-boht:r5:num}.
Combining \cref{eqn:special-boht:r5:num} with the robust non-reconstruction result, reconstruction is impossible below the KS threshold for $\lambda\in \left[ \lambda^*, \lambda_0\right]$.

\paragraph{Summary}
Combining the two cases, reconstruction is impossible below the KS threshold for $r=5$ and any $\lambda\in \left[\lambda^*, 1\right]$, where $\lambda^* = -\frac{1}{18}$.
This proves the $r=5$ case of \cref{thm:boht}.

\subsection{Case \texorpdfstring{$r=6$}{r=6}} \label{sec:special-boht:r6}
\begin{figure}[h!]
  \centering
	\includegraphics[width=0.4\textwidth]{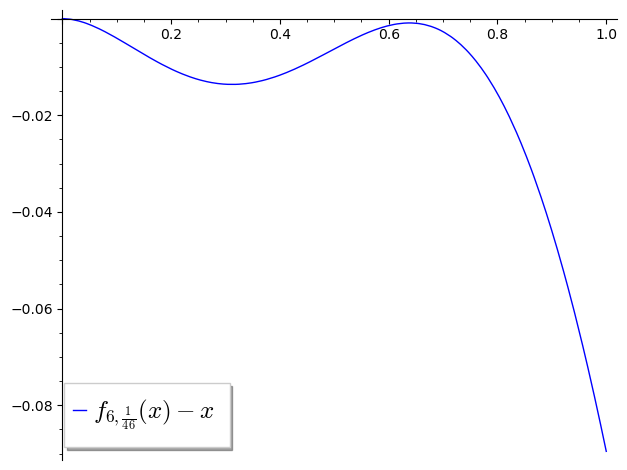}
  \caption{Plot of $f_{6,\frac 1{46}}(x)-x$ for $x\in [0, 1]$.}
  \label{fig:boht-r6-mid}
\end{figure}

Set $\lambda_0 = \frac 1{46}$. For any $\lambda\in [\lambda_0, 1]$ and $0<x\le 1$, we have
\begin{align}
  f_{r,\lambda}(x) \le f_{r,\lambda_0}(x) < x.
\end{align}
See \cref{fig:boht-r6-mid} for a plot of $f_{6,\frac 1{46}}(x)-x$.
Therefore reconstruction is impossible below the KS threshold for $r=6$ and $\lambda\in [\lambda_0, 1]$.
This proves the $r=6$ case of \cref{thm:boht}.



\section{Applications} \label{sec:app}
\subsection{Random NAE-SAT and random hypergraph bicoloring} \label{sec:app:nae-sat}
We first recall \cite{coja2017information}'s formula for the condensation threshold.
For a finite alphabet $\cX$, let $\scrP(\cX)$ denote the space of probability distributions over $\cX$, $\scrP^2(\cX)$ denote the space of probability distributions over $\scrP(\cX)$, and $\scrP^2_*(\cX)$ denote the space of probability distributions over $\scrP(\cX)$ with expectation $\Unif(\cX)$.
Define $\Lambda: \bR_{\ge 0} \to \bR_{\ge 0}$ as $\Lambda(x) = x\log x$ for $x>0$ and $\Lambda(0)=0$.
Let $\xi = 1-\frac{1-\exp(-\beta)}{2^{k-1}}$.
For $\mu \in \scrP^2_*(\{\pm\})$, define
\begin{align} \label{eqn:nae-sat-1rsb-pred-B}
  \scrB_{k,\beta,d}(\mu) =& \bE_{b\sim \Pois(d)} \bE_{(\pi_{i,j})_{i\in [b], j\in [k-1]} \iidsim \mu} \left[ \frac{\xi^{-b}}2 \Lambda\left( \sum_{x\in \{\pm\}} \prod_{i\in [b]} \left(1-\left(1-e^{-\beta}\right) \prod_{j\in [k-1]} \pi_{i,j}(x) \right) \right) \right] \\
  \nonumber &~- \bE_{\pi_1,\ldots,\pi_k \iidsim \mu} \left[ \frac{d(k-1)}{k \xi} \Lambda\left(1-\left(1-e^{-\beta}\right) \left( \sum_{x\in \{\pm\}}\prod_{i\in [k]} \pi_i(x)\right)\right) \right].
\end{align}
If $\mu$ is the point distribution at $\Unif(\{\pm\})$, then
\begin{align} \label{eqn:nae-sat-scrB-trivial}
  \scrB_{k,\beta,d}(\mu) = \log 2 + \frac dk \log \xi
\end{align}
is equal to the RS prediction \cref{eqn:nae-sat-rs-pred}.
\cite{coja2017information} proved that for a large class of random CSPs including random NAE-$k$-SAT and random hypergraph bicoloring, the condensation threshold $d_{k,\cond}(\beta)$ is given by
\begin{align} \label{eqn:nae-sat-cond}
  d_{k,\cond}(\beta) = \inf \left\{d > 0: \sup_{\mu \in \scrP^2_*(\{\pm\})} \scrB_{k,\beta,d}(\mu) > \log 2 + \frac dk \log \xi\right\}.
\end{align}

Evaluating the exact condensation threshold given by \cref{eqn:nae-sat-cond} requires solving an infinite-dimensional non-convex optimization problem. Local stability of the trivial solution $\mu = \mathbbm1_{\Unif(\{\pm\})}$ gives the Kesten-Stigum upper bound
\begin{align} \label{eqn:nae-sat-cond-ks}
  d_{k,\cond}(\beta) \le d_{k,\KS}(\beta) = \frac 1{k-1} \left( \frac{2^{k-1}-1+\exp(-\beta)}{1-\exp(-\beta)} \right)^2.
\end{align}
For large $k$ this bound is not tight: at zero temperature, \cite{coja2012condensation,sly2022number} showed
\begin{align} \label{eqn:nae-sat-cond-asymp}
  d_{k,\cond}(\infty) = \left(2^{k-1}-1\right) \log 2 + o_k(1).
\end{align}
Our small-arity result identifies regimes where the opposite behavior occurs and the Kesten-Stigum bound is exact.

The connection to reconstruction is as follows.
For any $\mu \in \scrP^2_*(\{\pm\})$, we can associate to it a binary-input channel $P=P^\mu: \{\pm\} \to \scrP(\{\pm\})$ via
\begin{align} \label{eqn:dist-to-channel}
  P(\cdot | x) = 2 \pi(x) d\mu(\pi), \qquad \forall x\in \{\pm\}.
\end{align}
Conversely, for any binary-input channel $P: \{\pm\}\to \cY$, we can associate to it a distribution $\mu=\mu^P \in \scrP^2_*(\{\pm\})$ via
\begin{align} \label{eqn:channel-to-dist}
  \mu = \bE_{x\sim \Unif(\{\pm\})} \bE_{y\sim P(\cdot | x)} \mathbbm1_{P_{X|Y} (\cdot | y)}.
\end{align}
In other words, $\mu$ is the law of the posterior distribution (under the uniform prior) $P_{X|Y}(\cdot |y)$, where $y\sim P(\cdot | x)$ and $x\sim \Unif(\{\pm\})$.
One checks that $\mu^{\left(P^\mu\right)}=\mu$ and $P^{\left(\mu^P\right)}$ is equivalent to $P$, so this establishes a bijection between $\scrP^2_*(\{\pm\})$ and the set of equivalence classes of binary-input channels.

\cite{coja2017information} (see also \cite{mezard2006reconstruction}) proved that to solve \cref{eqn:nae-sat-cond}, it suffices to optimize over fixed points of a function $\scrT: \scrP^2_*(\{\pm\}) \to \scrP^2_*(\{\pm\})$, defined as follows.
Let $\mu \in \scrP^2_*(\{\pm\})$ be the input. We define a one-layer $k$-uniform hypertree with root $\rho$.
\begin{enumerate}
  \item Generate $\sigma_\rho \sim \Unif(\{\pm\})$. This is the label of the root.
  \item Generate $b\sim \Pois(d)$. The root has $b$ downward hyperedges $(\rho,v_{i,1},\ldots,v_{i,k-1})_{i\in [b]}$.
  \item Independently for every hyperedge $(\rho,v_{i,1},\ldots,v_{i,k-1})$, generate $\sigma_{i,j}$ ($j\in [k-1]$) via
  \begin{align}
    (\sigma_{i,1},\ldots,\sigma_{i,k-1}) \sim B(\cdot | \sigma_\rho),
  \end{align}
  where $B$ is the broadcasting channel $B_{r,\lambda}$ (\cref{eqn:b-r-lambda}) with $r=k$ and
  \begin{align} \label{eqn:app:nae-sat:lambda-beta}
    \lambda = \frac{e^{-\beta}-1}{2^{k-1}-1+e^{-\beta}}.
  \end{align}
  \item Independently for every $v_{i,j}$, generate $\pi_{i,j}\in \scrP(\{\pm\})$ from $2 \pi(\sigma_{i,j}) d \mu(\pi)$.
  \item Compute $\pi_\rho\in \scrP(\{\pm\})$ via
  \begin{align}
    \pi_i(x) &= 1-(1-\exp(-\beta)) \prod_{j\in [k-1]} \pi_{i,j}(x), \qquad \forall i\in [b], x\in \{\pm\}\\
    \pi_\rho(x) &= \frac{\prod_{i\in [b]} \pi_i(x)}{\sum_{x'\in \{\pm\}} \prod_{i\in [b]} \pi_i(x')}, \qquad \forall x\in \{\pm\}.
  \end{align}
  \item The output $\scrT(\mu)$ is equal to the distribution of $\pi_\rho$ generated from the above process.
\end{enumerate}
This defines the operator $\scrT: \scrP^2_*(\{\pm\}) \to \scrP^2_*(\{\pm\})$.
Under the equivalence between $\scrP^2_*(\{\pm\})$ and binary-input channels, this corresponds to an operator that sends binary-input channels to binary-input channels.
Unfolding the definition, one obtains exactly the $\BP$ operator from \cref{eqn:intro:bp-operator}.
Therefore, to solve \cref{eqn:nae-sat-cond}, it suffices to optimize over binary-input fixed points of $\BP$.

\cref{thm:boht} shows that in the relevant regimes, the $\BP$ operator has no non-trivial BMS channel fixed point.
Let $P$ be a binary-input channel fixed point of $\BP$. Since $P\le_{\deg} \Id$, iterating the BP operator gives
\begin{align}
  P = \BP^\ell(P) \le_{\deg} \BP^{\ell}(\Id)
\end{align}
for any $\ell\ge 0$.
This implies
\begin{align}
  I\left(\pi, P\right) \le \lim_{\ell\to \infty} I\left(\pi, \BP^{\ell}(\Id)\right) = 0,
\end{align}
where $\pi=\Unif(\{\pm\})$.
So $P$ must be the trivial channel.
Thus the only fixed point of $\scrT$ in $\scrP^2_*(\{\pm\})$ is the trivial fixed point.

The preceding discussion shows that \cref{thm:boht} implies \cref{thm:nae-sat} via \cref{eqn:app:nae-sat:lambda-beta}.

\subsection{Hypergraph stochastic block model} \label{sec:app:hsbm}
\cref{thm:hsbm} follows from \cref{thm:boht} through the standard relationship between HSBM and BOHT.
\begin{theorem}[{\cite[Theorem 5.15]{gu2023channel}}] \label{thm:hsbm-boht}
  Consider the special HSBM model $\HSBM(n,r,a,b)$.
  Let $\BOHT(\Pois(d),r,\lambda)$ be the corresponding special BOHT, where $d$ and $\lambda$ are given by \cref{eqn:hsbm-boht-params}.
  If reconstruction is impossible for $\BOHT(\Pois(d),r,\lambda)$, then weak recovery is impossible for $\HSBM(n,r,a,b)$.
\end{theorem}

\section{Discussions} \label{sec:discuss}
\paragraph{Reconstruction threshold}
Known results on the reconstruction threshold for the special BOHT model $\BOHT(\Pois(d),r,\lambda)$ suggest that for every $r\ge 5$, there exists $\lambda_r^* \in \left[ -\frac 1{2^{r-1}-1}, 1\right]$ such that the KS threshold is tight for $\lambda>\lambda_r^*$ and is not tight for $\lambda<\lambda_r^*$. To our knowledge, the existence of such a $\lambda_r^*$ is not known for any $r\ge 5$.
In particular, for $r=6$, if such a $\lambda_r^*$ exists, then the KS threshold is tight for all $\lambda\in [0, 1]$.
The monotonicity of $f_{r,\lambda}$ in \cref{eqn:special-boht:f-r-lm} provides weak evidence for this conjecture.

A similar behavior has been conjectured for the Potts model with four states (\cite{ricci2019typology}).

\paragraph{Regular hypertrees}
We stated our main result \cref{thm:boht} for Poisson hypertrees because of their close relationship with random hypergraph models such as random NAE-SAT and HSBM\@.
Nevertheless, the methods in \cref{sec:robust,sec:pop} can provide non-trivial results for Galton-Watson hypertrees with general offspring distributions, including regular hypertrees, which may be connected to random regular hypergraph models.
However, our method does not prove some of the claims in \cref{thm:boht} if $\BOHT(\Pois(d),r,\lambda)$ is replaced with the regular hypertree model $\BOHT(d,r,\lambda)$.
For example, for $\BOHT(d,r,\lambda)$, with $r=3$, $d=5$, $\lambda=-\frac 1{\sqrt{10}}$, the $f$ function (\cref{eqn:prop:robust:f}) is expanding near $x=0$, i.e., $f(x)>x$ for all small enough $x>0$.
We know that robust non-reconstruction still holds by \cite{gu2024community}, so in this case our method is not tight.

\paragraph{Potts model}
One interesting question is whether our method can be generalized to non-binary models such as the Potts model, where it is predicted that the KS threshold is tight for $q=3$, any $\lambda$, and $q=4$, $\lambda>\lambda_4^*$ for some $\lambda_4^*<0$.
Our method relies heavily on the less-noisy preorder, which behaves very nicely for BMS channels but less well for more general channels. For example, the natural generalization of \cref{lem:bms-ln-coupling} to symmetric channels with a non-binary input alphabet does not always hold.
On the other hand, the degradation preorder still works nicely for $q$-ary symmetric channels (\cite{gu2023uniqueness}).
In principle, one could run rigorous population dynamics with the degradation preorder rather than the less-noisy preorder, and then use the robust non-reconstruction results proved using Sly's Taylor expansion method (\cite{sly2011reconstruction,mossel2023exact}). However, it is unclear how much computational power is needed to execute this framework. We leave this generalization to the Potts model as a further direction.

\ifdefined\isarxiv
\section*{Acknowledgments}
Part of this work was done while the author was supported by the National Science Foundation under Grant No.~DMS-1926686.
The author thanks Yury Polyanskiy and Allan Sly for helpful discussions.
\else
\fi

\bibliographystyle{alpha}
\bibliography{ref}

\appendix
\section{Verification Code}\label{app:code}
This appendix records the script and output for the rigorous population dynamics verification in \cref{sec:special-boht}. The script is implemented in SageMath 10.4.

\lstset{
  basicstyle=\ttfamily\footnotesize,
  breaklines=true,
  columns=fullflexible,
  keepspaces=true,
  numbers=left,
  numbersep=6pt,
  showstringspaces=false,
  xleftmargin=2em
}

\begin{lstlisting}[language=Python,frame=single]
# Data structure for building a BMS channel P while maintaining a BMS channel Q satisfying
# 1) P <=_ln Q
# 2) Q has bounded support
# Q is called an upper quantization of P
# c.f. Algorithm 2 in the paper
class BMSBuilder:
    # mesh: Q is supported on i/mesh for 0 <= i <= mesh
    # prec: all stored numbers are fractions with denominator prec
    def __init__(self, mesh, prec):
        self.mesh = mesh
        self.a = [0]*(self.mesh+1)
        self.prec = prec

    # normalize the channel so that total weight is 1
    def normalize_upper(self):
        s = sum(self.a)
        assert (s <= 1)
        self.a[-1] += 1-s

    # whether the channel is normalized
    def normalized(self):
        s = sum(self.a)
        return s == 1

    # helper function for adding p weight to k/mesh
    def __add_inner(self, k, p):
        assert (0 <= p <= 1)
        assert (0 <= k <= self.mesh)
        p_prec = floor(p*self.prec)/self.prec
        self.a[k] += p_prec

    # add BSC(theta=t) with weight p
    def add(self, t, p):
        assert (0 <= t <= 1)
        assert (0 <= p <= 1)
        if t == 1:
            self.__add_inner(self.mesh, p)
            return
        k = floor(t*self.mesh)
        assert (0 <= k <= self.mesh-1)
        tl = k/self.mesh
        tr = (k+1)/self.mesh
        ratio = (tr**2-t**2)/(tr**2-tl**2)
        self.__add_inner(k, p*ratio)
        self.__add_inner(k+1, p*(1-ratio))

    def items(self):
        return zip(self.a, [i/self.mesh for i in range(self.mesh+1)])

    # add BMS channel b with weight ratio
    def add_bms(self, b, ratio):
        for p, t in b.items():
            self.add(t, p*ratio)

    # compute chi2 information of the current channel
    def chi2_info(self):
        assert (self.normalized())
        s = sum(p*t**2 for p, t in self.items())
        return s

# Params
# r: hyperedge size
# lm: lambda, strength param
# d: expected offspring
# mesh: support size
# prec: precision
class Param:
    def __init__(self, r, lm, d, mesh=8, prec=2**32):
        self.r = r
        self.lm = lm
        self.d = d
        self.mesh = mesh
        self.prec = prec

# compute an upper quantization of the star product of two BMS channels
def star_bms(a, b, param):
    c = BMSBuilder(param.mesh, param.prec)
    for p1, t1 in a.items():
        for p2, t2 in b.items():
            p = p1*p2
            c.add((t1+t2)/(1+t1*t2), p*(1+t1*t2)/2)
            if not (t1 == t2 == 1):
                c.add(abs(t1-t2)/(1-t1*t2), p*(1-t1*t2)/2)
    c.normalize_upper()
    return c

# compute an upper quantization of BP of a BMS channel
def bp_bms(a, param):
    r = param.r
    lm = param.lm
    mesh = param.mesh
    prec = param.prec
    assert (a.mesh == mesh)
    c = BMSBuilder(mesh, prec)
    idx = [0]*(r-1)
    while True:
        p = 1
        t = []
        for i in range(r-1):
            t.append(idx[i]/a.mesh)
            p *= a.a[idx[i]]
        for w in range(2**(r-2)):
            tp = tn = 1
            for i in range(r-1):
                if ((w>>i)&1) == 1:
                    tp *= (1+t[i])
                    tn *= (1-t[i])
                else:
                    tp *= (1-t[i])
                    tn *= (1+t[i])
            c.add(abs(lm*(tp-tn)/(2*(1-lm)+lm*(tp+tn))), p*(2*(1-lm)+lm*(tp+tn))/(2**(r-1)))
        z = r-2
        while z >= 0 and idx[z] == mesh:
            idx[z] = 0
            z -= 1
        if z == -1:
            break
        idx[z] += 1
    c.normalize_upper()
    return c

# verify robust non-reconstruction for one set of parameters
# rob: target chi2 information. The goal is to reach chi2 info <= rob.
# stop: report failure if chi2 info stays above rob for stop iterations
# c.f. Algorithm 1 in the paper
def verify(param, rob, stop):
    d = param.d
    mesh = param.mesh
    prec = param.prec

    pois = [exp(-d)]
    D = 1
    while D <= d or pois[-1] >= 1/prec:
        pois.append(pois[-1]*d/D)
        D += 1
    for i in range(D):
        pois[i] = floor(pois[i]*prec)/prec

    a = BMSBuilder(mesh, prec)
    a.a[-1] = 1

    chi2_hist = [a.chi2_info().n()]
    it = 0
    print(f"iter {it} chi2_info {chi2_hist[-1]}", end='\r')

    while it < stop and chi2_hist[-1] > rob:
        bp_a = bp_bms(a, param)
        b = []
        b0 = BMSBuilder(mesh, prec)
        b0.a[0] = 1
        b.append(b0)
        b.append(bp_a)

        c = BMSBuilder(mesh, prec)
        for i in range(D):
            if i >= 2:
                b.append(star_bms(b[i//2], b[i-i//2], param))
            c.add_bms(b[i], pois[i])
        c.normalize_upper()
        a = c
        chi2_hist.append(a.chi2_info().n())
        it += 1
        print(f"iter {it} chi2_info {chi2_hist[-1]}", end='\r')
    return (chi2_hist[-1] <= rob, it)

# verify tightness of KS for all lambda between lm_s and lm_t
def verify_all(r, lm_s, lm_t, prec, rob, fixed_step=None):
    lm = lm_s
    it = oo
    step = 1 if fixed_step is None else fixed_step
    while lm < lm_t:
        if fixed_step is None and it < 30:
            step += 1
        d = 1/((r-1)*lm**2) + step
        f, it = verify(Param(r=r, lm=lm, d=d), rob, stop=100)
        assert(f)
        print(f"r={r} lm={n(lm)} d={n(d)} it={it}")
        lm = floor(-1/sqrt((r-1)*d)*prec)/prec
    print(lm)

# verify tightness of KS for r=4
# c.f. Table 1 and Figure 2 in the paper
verify_all(r=4, lm_s=-1/7, lm_t=-13/100, prec=1000, rob=2/10, fixed_step=1)

# verify tightness of KS for r=5
# c.f. Table 2 and Figure 3 in the paper
verify_all(r=5, lm_s=-1/18, lm_t=-1/37, prec=10000, rob=2/10)
\end{lstlisting}

The following captured output from the script verifies \cref{tab:boht-r4,tab:boht-r5}.
\begin{lstlisting}[frame=single]
r=4 lm=-0.142857142857143 d=17.3333333333333 it=47
r=4 lm=-0.139000000000000 d=18.2523851422459 it=33
r=4 lm=-0.136000000000000 d=19.0219146482122 it=27
r=4 lm=-0.133000000000000 d=19.8441027380481 it=24
-13/100
r=5 lm=-0.0555555555555556 d=82.0000000000000 it=68
r=5 lm=-0.0553000000000000 d=82.7503735992074 it=61
r=5 lm=-0.0550000000000000 d=83.6446280991736 it=54
r=5 lm=-0.0547000000000000 d=84.5536364213643 it=50
r=5 lm=-0.0544000000000000 d=85.4777249134948 it=46
r=5 lm=-0.0541000000000000 d=86.4172289967576 it=43
r=5 lm=-0.0538000000000000 d=87.3724934702395 it=41
r=5 lm=-0.0535000000000000 d=88.3438728273212 it=39
r=5 lm=-0.0532000000000000 d=89.3317315846006 it=37
r=5 lm=-0.0530000000000000 d=89.9996440014240 it=36
r=5 lm=-0.0528000000000000 d=90.6751606978880 it=35
r=5 lm=-0.0526000000000000 d=91.3583975480345 it=35
r=5 lm=-0.0524000000000000 d=92.0494726414545 it=34
r=5 lm=-0.0522000000000000 d=92.7485063343169 it=33
r=5 lm=-0.0520000000000000 d=93.4556213017752 it=32
r=5 lm=-0.0518000000000000 d=94.1709425917920 it=32
r=5 lm=-0.0516000000000000 d=94.8945976804279 it=31
r=5 lm=-0.0514000000000000 d=95.6267165286378 it=30
r=5 lm=-0.0512000000000000 d=96.3674316406250 it=30
r=5 lm=-0.0510000000000000 d=97.1168781237985 it=29
r=5 lm=-0.0508000000000000 d=98.8751937503875 it=32
r=5 lm=-0.0503000000000000 d=100.810714243367 it=31
r=5 lm=-0.0498000000000000 d=102.804825728617 it=30
r=5 lm=-0.0494000000000000 d=104.443901719418 it=29
r=5 lm=-0.0490000000000000 d=107.123281965848 it=31
r=5 lm=-0.0484000000000000 d=109.720852400792 it=29
r=5 lm=-0.0478000000000000 d=113.416851945869 it=31
r=5 lm=-0.0470000000000000 d=117.173381620643 it=29
r=5 lm=-0.0462000000000000 d=122.126740503364 it=30
r=5 lm=-0.0453000000000000 d=126.827015384316 it=28
r=5 lm=-0.0444000000000000 d=132.816005194384 it=29
r=5 lm=-0.0434000000000000 d=139.727388562085 it=29
r=5 lm=-0.0423000000000000 d=147.720224223016 it=29
r=5 lm=-0.0412000000000000 d=156.280610802149 it=29
r=5 lm=-0.0400000000000000 d=166.250000000000 it=28
r=5 lm=-0.0388000000000000 d=177.064406419386 it=28
r=5 lm=-0.0376000000000000 d=188.833408782254 it=27
r=5 lm=-0.0364000000000000 d=201.684941432194 it=26
r=5 lm=-0.0353000000000000 d=214.627563017118 it=26
r=5 lm=-0.0342000000000000 d=228.740980130638 it=25
r=5 lm=-0.0331000000000000 d=244.183386424001 it=24
r=5 lm=-0.0320000000000000 d=261.140625000000 it=24
r=5 lm=-0.0310000000000000 d=278.145681581686 it=23
r=5 lm=-0.0300000000000000 d=296.777777777778 it=22
r=5 lm=-0.0291000000000000 d=315.225611412241 it=22
r=5 lm=-0.0282000000000000 d=335.370504501786 it=21
r=5 lm=-0.0274000000000000 d=354.995897490543 it=21
-133/5000
\end{lstlisting}





\end{document}